\documentclass{amsart}

\usepackage{graphicx}
\usepackage[all]{xy}
\usepackage{color}
\usepackage{amsmath}
\usepackage{amsfonts}
\usepackage{amsthm}
\usepackage{amssymb}
\usepackage{mathrsfs}

\theoremstyle{theorem}
\newtheorem{thrm}{Theorem}
\newtheorem{lem}[thrm]{Lemma}
\newtheorem{cor}[thrm]{Corollary}
\newtheorem{prop}[thrm]{Proposition}
\newtheorem{conj}[thrm]{Conjecture}
\newtheorem{quest}[thrm]{Question}

\numberwithin{thrm}{subsection} 

\def \Dj{\mbox{\raise0.3   ex\hbox{-}\kern-0.4em D}}

\newcommand{\Crit}{\operatorname{Crit}}

\newcommand{\im}{\operatorname{im}}

\newcommand{\id}{{\bf{1}}}

\theoremstyle{remark}
\newtheorem{rem}[thrm]{Remark}
\newtheorem{exm}[thrm]{Example}
\newtheorem{defn}[thrm]{Definition}

\begin{document}

\title[Persistence barcodes and Laplace eigenfunctions]{Persistence barcodes and  Laplace eigenfunctions on surfaces}

\author[Iosif Polterovich]{Iosif Polterovich\textsuperscript{1}}
\address{D\'e\-par\-te\-ment de math\'ematiques et de
sta\-tistique, Univer\-sit\'e de Mont\-r\'eal,  CP 6128 succ
Centre-Ville, Mont\-r\'eal,  QC  H3C 3J7, Canada.}
\email{iossif@dms.umontreal.ca}
\author[Leonid Polterovich]{ Leonid Polterovich\textsuperscript{2}}
\address{School of Mathematical Sciences, Tel Aviv University, Ramat Aviv, Tel Aviv 69978
Israel}
\email{polterov@post.tau.ac.il}
\author[Vuka\v sin Stojisavljevi\'c ]{Vuka\v sin Stojisavljevi\'c\textsuperscript{2}}
\address{School of Mathematical Sciences, Tel Aviv University, Ramat Aviv, Tel Aviv 69978
Israel}
\email{vukasin@post.tau.ac.il}
\footnotetext[1]{Partially supported by
NSERC, FRQNT and Canada Research Chairs program.}
\footnotetext[2]{Partially supported by the European Research Council Advanced grant 338809.}
\maketitle

\begin{abstract}
We obtain restrictions on the persistence barcodes of Laplace-Beltrami eigenfunctions and their linear combinations on compact surfaces with Riemannian metrics. Some applications to uniform approximation  by linear combinations of Laplace eigenfunctions are also discussed.
\end{abstract}

\tableofcontents

\section{Introduction and main results}
\subsection{Laplace-Beltrami eigenfunctions}
The past fifteen years have witnessed a number of fascinating applications of the spectral theory of the Laplace-Beltrami operator  to data analysis, such as dimensionality reduction and data representation
\cite{BN,CL} or shape segmentation in computer graphics \cite{SOCG,Re}. In the present paper we focus on this
interaction the other way around and study persistence barcodes, a fundamental notion originated in topological data analysis, of the Laplace-Beltrami eigenfunctions and their linear combinations. Our main finding is a constraint on such barcodes in terms of the corresponding eigenvalues. This result turns out to
have applications to approximation theory.

Let $M$ be a compact $n$-dimensional Riemannian manifold, possibly with nonempty boundary. Let $\Delta$ be the (positive definite) Laplace-Beltrami operator on $M$; if $\partial M \neq \emptyset$ we assume that the Dirichlet condition is imposed on the boundary.
The spectrum of the Laplace-Beltrami operator on a compact Riemannian manifold is discrete, and the eigenvalues form a sequence $0\le \lambda_1 \le \lambda_2\le  \dots \nearrow \infty$, where each eigenvalue is repeated according to its multiplicity. The corresponding eigenfunctions $f_k$, $\Delta f_k=\lambda_k f_k$, form  an orthonormal basis in $L^2(M)$.
The  properties of  Laplace-Beltrami  eigenfunctions have fascinated researchers for more than two centuries, starting with the celebrated Chladni's experiments with vibrating plates.
We refer to \cite{JNT, Z1, Z2} for a modern overview of the subject.   As the examples of trigonometric polynomials and spherical harmonics indicate, the shapes of the eigenfunctions are expected  to have an increasingly complex structure  as $\lambda$ goes to infinity.  At the same time, various  restrictions on the behaviour of eigenfunctions can be formulated in terms of the corresponding eigenvalue.   One of the basic facts about eigenfunctions is  Courant's nodal domain theorem,  stating that the number of nodal domains of an eigenfunction $f_k$  is at most $k$ (see \cite{CH}).
There exist also bounds on the $(n-1)$-dimensional measure of the zero set of eigenfunctions (see \cite{L1, L2, ML} for most recent developments on this topic),
on the distribution of nodal extrema (\cite{PolSod, Po}), on the growth of $L^p$-norms (\cite{S}),  and other related results.

In the present paper we focus on topological properties of the sublevel sets of Laplace-Beltrami eigenfunctions, and, more generally, of the linear combinations of eigenfunctions with  eigenvalues $\leq \lambda$. There has been a number of important recent advances in the study of  topological properties of random linear combinations of Laplace eigenfunctions, with an emphasis on the nodal and critical sets (see, for instance, \cite{NS, Nic, GW, GW2,  SW, CS}). Our approach is deterministic and is based on the study of {\it  persistence barcodes}.  In the probabilistic setting, some steps  in this direction have been discussed  in  \cite[Section 1.4.3]{CMW}, see also \cite{PauSt}. Roughly speaking, a persistence barcode is a collection of intervals in $\mathbb{R}$ which encodes oscillation of a function (see next subsection for a detailed overview). Our main result (Theorem \ref{Persistence_Bound}) implies that the quantity $\Phi_1(f)$, the total length of the barcode of any such linear combination $f$ with unit $L^2$-norm, satisfies an upper bound $O(\lambda)$. This inequality is inspired by the ideas introduced in \cite{PolSod}, where a similar bound was proved for the {\it Banach indicatrix} of $f$, another measure of oscillation which goes back to the works of Kronrod \cite{Kron} and Yomdin \cite{Yom}. Our central observation  (see Proposition~\ref{Bound_Topological} below) is that the length of the barcode admits an upper bound via the Banach indicatrix, which together with \cite{PolSod} yields the main result.

We believe that discussing eigenfunctions and their linear combinations in the language of barcodes, which originated in topological data analysis, has a number of merits. First, there exists a well developed metric theory of barcodes which highlights their robustness with respect to perturbations of functions in the uniform norm. Some features of this robustness are inherited by the above-mentioned functional $\Phi_1$.  This, in turn, paves the way for applications to the following question of approximation theory (see Section~\ref{section-approximations by eigenfunctions}): given a function with unit $L^2$-norm, how well one can approximate it by a linear combination of Laplace eigenfunctions  with  eigenvalues $\leq \lambda$?  In particular, we show that a highly oscillating function does not admit a good uniform approximation of this kind unless $\lambda$ is large enough, see  Corollary \ref{214}.  Second, our approximation results remain valid  if a given function is composed  with a diffeomorphism of the surface, see Proposition \ref{211}. Our approach yields  it essentially for free, given that the barcodes are invariant with respect to compositions with diffeomorphisms. Note that the effect of a change of variables on analytic properties of functions is a classical theme in Fourier analysis, cf. the celebrated Bohr-P\'al theorem \cite{BP}.   Third, we conjecture that barcodes provide a right framework for a potential extension of our results to higher dimensions, see Conjecture \ref{higherdimconj} below.

In a different direction, we present an application to the problem of sorting finite bars of persistence barcodes. This task arises
on a number of  occasions in topology and data analysis. Our results allow to improve an estimate on the optimal running time of a sorting algorithm for barcodes of linear combinations of Laplace eigenfunctions  with  eigenvalues $\leq \lambda$, see subsection \ref{sort}.
\subsection{Persistence modules and barcodes}
\label{pers}
In order to describe the topology of the sublevel sets of eigenfunctions and their linear combinations,  we use the notions of persistence modules and barcodes,  which are briefly reviewed below. We refer the reader to articles \cite{EH,G,Ca,W,BauerLesnick} and monographs \cite{E,Ou,pers-book} for an introduction into this rapidly developing subject and further details.

Let $M$ be a closed, connected, orientable, $n$-dimensional manifold, possibly with boundary and $f:M\rightarrow \mathbb{R}$ a Morse function (if $\partial M\neq \emptyset$, we assume that $f$ is Morse on $M \setminus \partial M$ and $f\vert_{\partial M} =const$, the value on the boundary being regular). For $k=0,\ldots,n$ define a family of finite dimensional real vector spaces depending on a parameter $t\in \mathbb{R}$ by setting
$$ V^t_k(f)=H_k(f^{-1}((-\infty,t));\mathbb{R}), $$
where $H_k(X; \mathbb{R})$ denotes the $k$-th homology with coefficients in $\mathbb{R}$ of a set $X$. If $s\leq t$ we have that $f^{-1}((-\infty,s))\subset f^{-1}((-\infty,t))$ and the inclusion of sublevel sets $i_{st}:f^{-1}((-\infty,s))\rightarrow f^{-1}((-\infty,t))$ induces a map
$$\pi_{st}=(i_{st})_*: V^s_k(f) \rightarrow V^t_k(f) \text{ for } k=0,\ldots,n.$$
These maps are called {\it comparison maps}. The family of vector spaces $V^t_k(f)$ together with the family of comparison maps $\pi_{st}$ forms an algebraic structure called a {\it persistence module}. We call this persistence module a degree $k$ persistence module associated to $f$. We define the \textit{spectrum} of $V_k^t$ as those points $r\in \mathbb{R}$ for which $V_k^t$ "changes when $t$ passes through $r$". More formally, if $r$ has a neighbourhood $U_r$ such that $\pi_{st}$ are isomorphisms for all $s,t\in U_r$, we say that $r$ is not inside spectrum and we say that $r$ is inside spectrum otherwise. One may readily check that under the above assumptions spectrum is finite and consists of critical values of $f$ (and possibly also the value of $f$ on the boundary of $M$ if $\partial M \neq \emptyset$). Moreover, if $a<b$ are consecutive points of the spectrum, then $\pi_{st}$ is an isomorphism for all $a<s<t\leq b$.

Persistence modules that we will consider will be associated to a certain function $f:M\rightarrow \mathbb{R}$ as described above. However, one may define a persistence module as an abstract algebraic structure, without the auxiliary function $f$. In this setting, a persistence module consists of a one-parametric family of finite dimensional\footnote[3]{These are sometimes refered to as {\it pointwise finite dimensional} persistence modules.} vector spaces $V^t,~t\in \mathbb{R}$ and a family of linear comparison maps $\pi_{st}:V^s\rightarrow V^t$ for $s\leq t$, which satisfy $\pi_{tt}=\id_{V^t}$ and $\pi_{st}\circ \pi_{rs}= \pi_{rt}$ for all real numbers $r\leq s \leq t.$ Some of the results that we will refer to, namely the {\it structure theorem} (whose early version appeared in \cite{Ba})  and the {\it isometry theorem}, may be formulated and proven in this abstract language.

According to the structure theorem for persistence modules (under some extra assumptions which hold for the persistence modules that  we consider), $V^t$ decomposes into a direct sum of simple persistence modules. Let $I=(a,b]$ or $I=(a,+\infty)$, $a,b \in \mathbb{R}$ and denote by $Q(I)=(Q(I),\pi)$ the persistence module which satisfies $Q^t(I)=\mathbb{R}$ for $t\in I$ and $Q^t(I)=0$ otherwise and $\pi_{st}=id$ for $s,t\in I$ and $\pi_{st}=0$ otherwise. Now
$$V^t \cong \bigoplus_i (Q^t(I_i))^{m_i},$$
where $m_i$ are finite multiplicities, $I_i$ are intervals of the form $(a_i,b_i]$ or $(a_i,+\infty)$ and $(Q(I),\pi)\oplus(Q(I'),\pi')=(Q(I)\oplus Q(I'),\pi\oplus\pi')$. The above decomposition is unique if we assume that $I_i\neq I_j$ when $i\neq j$. The multiset containing $m_i$ copies of $I_i$ is called the {\it barcode} of $V^t$ and is denoted by $\mathcal{B}(V^t)$ and intervals $I_i$ are called {\it bars}. In the case of a module $V^t=V_k^t(f)$ coming from a function $f$, we have that endpoints $a_i$, $b_i$ of all of bars belong to the spectrum of $V_k^t$ and the barcode $\mathcal{B}(V_k^t)=\mathcal{B}_k(f)$ is called the {\it degree $k$ barcode} of $f$. We also denote by $\mathcal{B}(f)=\cup_k \mathcal{B}_k(f)$ the {\it full barcode} of $f$. Under our assumptions $\mathcal{B}(f)$ is finite, i.e. it consists of finitely many distinct intervals with finite multiplicities.
{\exm Let $f:\mathbb{S}^1\rightarrow \mathbb{R}$ be a height function on a deformed circle (see Figure 1).
\begin{figure}
\begin{center}
\includegraphics[width=8cm]{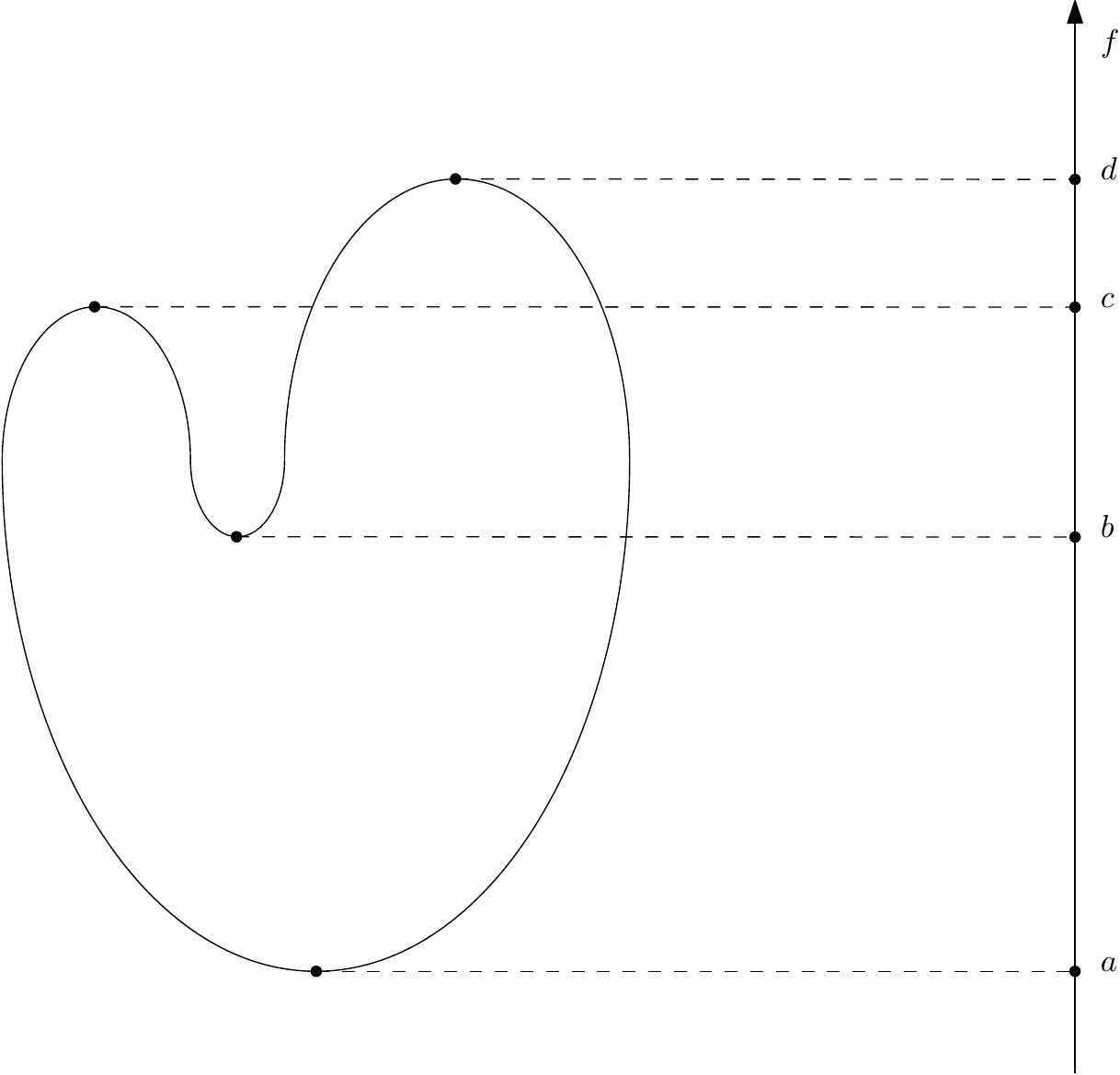}
\caption{Height function on a deformed circle.}
\end{center}
\end{figure}
Critical values of $f$ are $a,b,c$ and $d,$ and for a regular value $t\in \mathbb{R},$ the sublevel sets $f^{-1}((-\infty, t))$ are homeomorphic to:
$$f^{-1}((-\infty, t))=\begin{cases}
      \emptyset, & \text{for}\ t<a \\
      I, & \text{for}\ a<t<b \\
      I\sqcup I, & \text{for}\ b<t<c \\
      I, & \text{for}\ c<t<d \\
      \mathbb{S}^1, & \text{for}\ d<t \\
    \end{cases}$$
where $I$ stands for an open interval. Degree 1 barcode is now easily seen to contain one infinite bar $\mathcal{B}_1(f)=\{ (d,+\infty) \},$ while degree 0 barcode contains one infinite and one finite bar $\mathcal{B}_0(f)=\{ (a,+\infty),(b,c] \}.$ The finite bar $(b,c]$ corresponds to the fact that for $b<t<c$, $f^{-1}((-\infty, t))$ has two connected components which merge for $t>c.$ The full barcode is given by $\mathcal{B}(f)=\{(a,+\infty),(d,+\infty),(b,c]\}.$}
\\
\\
Let us introduce some notation and give a short description of the barcode associated to a function $f:M \rightarrow \mathbb{R}$ under the above assumptions. Denote by $b_i=\dim H_i(M;\mathbb{R}),~i=0,\ldots,n$ the Betti numbers of $M$. Bars of infinite lengths in degree $k$ barcode correspond to homology classes in $H_k(M;\mathbb{R})$. This means that the number of infinite bars in degree $k$ barcode is equal to $b_k$. Since $M$ is connected and orientable, we know that $b_0=1$, while $b_n=1$ if $\partial M = \emptyset$ and $b_n=0$ otherwise. Thus, if $\partial M = \emptyset$, degree 0 and degree $n$ barcodes each contain exactly one infinite bar $(c^{(0)},+\infty)$ and $(c^{(n)},+\infty)$, while degree $k$ barcode contains $b_k$ infinite bars $(c_1^{(k)},+\infty),\ldots ,(c_{b_k}^{(k)},+\infty)$, $c_1^{(k)} \leq \ldots \leq c_{b_k}^{(k)}$ for $1\leq k \leq n-1$. If $\partial M \neq \emptyset$, the barcode is similar, the only difference being that there are no infinite bars in degree $n$. Numbers $c_i^{(k)}$, $1\leq i \leq b_k$ are called \textit{spectral invariants} of a function $f$ in degree $k$ and it holds $c^{(0)}=\min f.$ If $\partial M=\emptyset$, we also have that $c^{(n)}=\max f.$ All finite bars are contained in $(\min f,\max f]$ and there are no finite bars in degree $n.$

The space of barcodes can be endowed with a natural distance. Let $\mathcal{B}$ and $\tilde{\mathcal{B}}$ be two barcodes. We say that $\mu:U\rightarrow \tilde{U}$ is an $\varepsilon$-matching between $\mathcal{B}$ and $\tilde{\mathcal{B}}$ if it is a bijection between some subsets $U\subset \mathcal{B} $ and $\tilde{U}\subset \tilde{\mathcal{B}}$ which contain all bars of length greater than $2\varepsilon$ and it satisfies
$$\mu((a,b])=(c,d] \Rightarrow |a-c|,|b-d|\leq \varepsilon.$$
Intuitively, we may delete some bars of length less or equal than $2\varepsilon$ from $\mathcal{B}$ and $\tilde{\mathcal{B}}$, and match the rest up to an error of $\varepsilon$. We may also think of the deleted bars as matched with intervals of length 0. Now, the {\it bottleneck distance} between $\mathcal{B}$ and $\tilde{\mathcal{B}}$, $d_{bottle}(\mathcal{B},\tilde{\mathcal{B}})$, is defined as infimum over all $\varepsilon \geq 0$ such that there exists an $\varepsilon$-matching between $\mathcal{B}$ and $\tilde{\mathcal{B}}$. One of the direct consequences of the isometry theorem is the following stability result which we will need later. For two functions $f$ and $g$ it holds
\begin{equation}\label{Stability}
  d_{bottle}(\mathcal{B}(f),\mathcal{B}(g)) \leq |f-g|_{C^0},
\end{equation}
where $|f-g|_{C^0}=\max\limits_x |f(x)-g(x)|$.
{\rem Full barcodes $\mathcal{B}(f)$ and $\mathcal{B}(g)$ which we consider do not keep track of the degrees. In particular, this means that we allow matchings between $\mathcal{B}(f)$ and $\mathcal{B}(g)$ to match bars from $\mathcal{B}_i(f)$ with bars from $\mathcal{B}_j(g)$ for $i\neq j.$ The isometry theorem can also be applied in each degree independently to obtain
$$\max_k d_{bottle}(\mathcal{B}_k(f),\mathcal{B}_k(g)) \leq |f-g|_{C^0}.$$
From here (\ref{Stability}) immediately follows because
$$d_{bottle}(\mathcal{B}(f),\mathcal{B}(g)) \leq \max_k d_{bottle}(\mathcal{B}_k(f),\mathcal{B}_k(g)).$$}

\subsection{A family of functionals on the space of barcodes}
From now on, we assume that $M$ is an orientable surface, possibly with boundary. Let us define, for every positive function $u\in C(\mathbb{R})$, a positive, lower semi-continuous functional $\Phi_u$ on the space of Morse  functions on $M$.
 Let $f$ be a Morse function, denote by $\mathcal{B'}(f)\subset \mathcal{B}(f)$ the multiset of all finite bars in the barcode $\mathcal{B}(f)$ and by $|\mathcal{B'}(f)|$ the total number of finite bars in $\mathcal{B}(f)$.
Define a positive functional $\Phi_u$ on the set
$\mathcal{F}_{Morse}$ of all Morse functions (vanishing on the boundary) by setting
\begin{equation}\label{Functional}
  \Phi_u(f)=\begin{cases}
     \int\limits_{\min f}^{\max f} u(t)~dt + \sum\limits_{I\in \mathcal{B'}(f)} \int\limits_I u(t)~dt & \mbox{if}~ \partial M = \emptyset, \\
     \\
     \int\limits_{\min f}^{0} u(t)~dt + \sum\limits_{I\in \mathcal{B'}(f)} \int\limits_I u(t)~dt &  \mbox{if }\partial M \neq \emptyset.
  \end{cases}
\end{equation}
In particular, $\Phi_1(f)$ is the sum of the lengths of all the finite bars in the barcode of $f$ and the length of the range of $f$. A related functional has been earlier considered in \cite{CSEHM}, see Remark \ref{remark:L^p}.


\begin{lem}
\label{lemma:semicon}
Let
$$C(u,f)=2\cdot (|\mathcal{B}'(f)| + 1) \cdot \max_{[\min f, \max f]} u $$
in the case $\partial M = \emptyset,$ or
$$C(u,f)= (2 |\mathcal{B}'(f)| + 1) \cdot \max_{[\min f, \max f]} u $$
in the case  $\partial M \neq \emptyset.$
Then
\begin{equation}\label{LSC_Barcode}
\Phi_u(f)-\Phi_u(h) \leq C(u,f) \cdot d_{bottle}(\mathcal{B}(f),\mathcal{B}(h)).
\end{equation}
\end{lem}
\begin{rem}\label{Critical_Points}
If $M$ has no boundary, then
\begin{equation}
\label{critpoints}
2\cdot (|\mathcal{B}'(f)| + 1)=|\Crit(f)|-b_1(M),
\end{equation}
where $|\Crit(f)|$ stands for the number of critical points of $f$ and $b_1=\dim H_1(M;\mathbb{R}).$ Morally speaking, each critical point of index $i$ produces either a left endpoint of a bar in degree $i$ or a right endpoint of a bar in degree $i-1$. This can be made precise in a number of ways  (see, for example,  \cite{W},  or \cite{UZ} for a  more general statement).  Therefore, taking into account the critical points corresponding to infinite bars, we get $|\Crit(f)|=2 |\mathcal{B}'(f)|+b_0(M)+b_1(M)+b_2(M)$, which implies \eqref{critpoints}. The same reasoning applies to general manifolds without boundary,  where we have
$$|\Crit(f)|=2 |\mathcal{B}'(f)| + \sum_{i=0}^{\dim M}b_i(M).$$
\end{rem}
We prove Lemma \ref{lemma:semicon} in subsection \ref{subs:semicon}. Combining (\ref{LSC_Barcode}) with (\ref{Stability}) yields
\begin{equation}\label{LSC_Functions}
\Phi_u(f)-\Phi_u(h)  \leq  C(u,f) \cdot d_{C^0}(f,h),
\end{equation}
where $d_{C^0}(f,h)=|f-h|_{C^0}.$
\begin{prop}
\label{lowersemicont}
The functional  $\Phi_u$ is lower semi-continuous both as a functional
$\Phi_u:(\mathbf{B},d_{bottle})\rightarrow \mathbb{R}$ and as a functional  $\Phi_u:(\mathcal{F}_{\text{Morse}},d_{C^0})\rightarrow \mathbb{R}$.
Here $\mathbf{B}$ stands for the set of all barcodes corresponding to functions in $\mathcal{F}_{\text{Morse}}.$
\end{prop}
\begin{rem}
We slightly abuse the notation here by looking at $\Phi_u(f)=\Phi_u(\mathcal{B}(f))$ as the function of barcode $\mathcal{B}(f)$. However, it is obvious that $\Phi_u$ depends only on $\mathcal{B}(f)$ and not on $f$ itself. In the same spirit $\min f$ and $\max f$ should be replaced by the smallest and the largest endpoint of a bar in $\mathcal{B}(f).$
\end{rem}
\begin{proof} Recall that a functional $\Phi$ defined on a metric space $X$ is called lower semi-continuous at a point $f \in X$ if $\liminf_{h\to f} \Phi(h) \ge \Phi(f).$
This relation easily follows from the inequalities  (\ref{LSC_Barcode}) and (\ref{LSC_Functions}) for the functional $\Phi_u$ defined on the metric spaces $(\mathbf{B},d_{bottle})$ and  $(\mathcal{F}_{Morse},d_{C^0})$, respectively.
\end{proof}
The inequality  (\ref{LSC_Functions}) could be further strengthened. Let $\text{Diff}(M)$ denote the group of all smooth diffeomorphisms of the surface $M$ (throughout the paper, the term ``smooth''  stands for $C^{\infty}$--smooth).
\begin{cor}
\label{useful}
We have
\begin{equation}\label{LSC_Diffeo}
\Phi_u(f)-\Phi_u(h) \leq \  C(u,f) \cdot d_{C^0}(f \circ \varphi,h \circ \psi),
\end{equation}
for any two diffeomorphisms  $\varphi,\psi \in \text{Diff}(M).$
\end{cor}
In particular, taking  $\varphi=\psi=\id_M$ gives (\ref{LSC_Functions}).
\begin{proof}
Indeed,  for any diffeomorphism $\varphi:M\rightarrow M$,  the  barcodes $\mathcal{B}(f)$ and $\mathcal{B}(f\circ \varphi)$ are the same. Since $\Phi_u$ depends only on the barcode and not on the function itself, putting $f\circ \varphi$ and $h \circ \psi$ in (\ref{LSC_Barcode}) yields \eqref{LSC_Diffeo}.
\end{proof}
Let us now extend the functional $\Phi_u$ from $\mathcal{F}_{Morse}$ to $C^0(M)$.   First, we introduce a ``cut-off version'' of $\Phi_u$.
 Define
\begin{equation}\label{Functionalk}
  \Phi_{u,k}(f)=\begin{cases}
     \int\limits_{\min f}^{\max f} u(t)~dt + \sum\limits_{i=1}^k \int\limits_{I_i} u(t)~dt  & \mbox{if}~ \partial M = \emptyset, \\
     \\
     \int\limits_{\min f}^{0} u(t)~dt + \sum\limits_{i=1}^k \int\limits_{I_i} u(t)~dt &  \mbox{if }\partial M \neq \emptyset.
  \end{cases}
\end{equation}
where $I_i\in \mathcal{B'}(f)$ are finite intervals ordered by integral of $u$, i.e. we have
$$\int\limits_{I_1}u(t)~dt\geq \int\limits_{I_2}u(t)~dt\geq \ldots$$
\begin{lem}
\label{Lipsch}
For every bounded function $u$ the functional  $\Phi_{u,k}$ is Lipschitz on  $\mathcal{F}_{\text{Morse}}$  with respect to $d_{bottle}$ with Lipschitz constant $(2k+2)\cdot \max u$ if $\partial M=\emptyset$ or with Lipschitz constant $(2k+1)\cdot \max u$ if $\partial M\neq\emptyset$.
\end{lem}
The proof of Lemma \ref{Lipsch} is given in subsection \ref{lemmaproof}.

\smallskip

 Assume now that $f\in C^0(M)$ is  an abritrary continuous function on $M$.  Let $f_n\in \mathcal{F}_{Morse}$ such that $d_{C^0}(f,f_n)\to 0$ as $n\to \infty$. Set
\begin{equation}
\label{nonmorsefunct}
\Phi_u(f):=\lim_{k\to\infty}\lim_{n\to\infty} \Phi_{u,k}(f_n)
\end{equation}
Note that $\Phi_u(f_n)$ only depends on $u|_{[\min f_n, \max f_n]}$ and, since for sufficiently large $n$ it holds $[\min f_n, \max f_n ] \subset [\min f -1, \max f+1]$, we may restrict ourselves to this interval and argue as if $u$ was bounded. Thus due to Lemma \ref{Lipsch} and \eqref{Stability}, the double limit on the left hand side of \eqref{nonmorsefunct} (which could be equal to $+\infty$) does not depend on the choice of the approximating sequence
$f_n$. Therefore, the functional $\Phi_u(f)$ is well defined by \eqref{nonmorsefunct}.  Moreover, it is easy to check that the right-hand sides of \eqref{nonmorsefunct} and \eqref{Functional} coincide for $f\in \mathcal{F}_{Morse}$, and therefore \eqref{nonmorsefunct} indeed defines an extension of \eqref{Functional} to $C^0(M)$.

\subsection{Main results}\label{subsec-mr} As before, $M$ is an orientable surface, possibly with boundary, equipped with a Riemannian metric $g$. Denote by $\| \cdot \|$ the $L^2$-norm with respect to Riemannian area $\sigma$ and by $\Delta$ the Laplace-Beltrami operator with respect to $g$. Slightly abusing the notation, throughout the paper $\kappa_g$ will denote various constants depending only on the Riemannian metric $g$.

Following\footnote[4]{Our definition is slightly different from the one in \cite{PolSod} since we do not assume that $\int_{M} f~\sigma=0$ if $M$ has no boundary. However, this assumption is not needed for any of the results of \cite{PolSod} which we use.} \cite{PolSod}, denote by $\mathcal{F}_\lambda$  the set of all smooth functions on $M$ (vanishing on the boundary if $\partial M \neq 0$) which satisfy $\| f \|=1$ and $\| \Delta f \| \leq \lambda.$ One may check that $\mathcal{F}_\lambda$ contains normalized linear combinations of eigenfunctions of $\Delta$ with eigenvalues $\lambda_i \leq \lambda$. If  $\partial M \neq 0$,  $\mathcal{F}_\lambda$ contains also normalized eigenfunctions of the biharmonic clamped plate boundary value problem on $M$
(see \cite[Example 1.2]{PolSod}). Our main result is the following theorem.
\begin{thrm}
\label{Persistence_Bound}
Let $\lambda>0$ be any positive real number, $u\in C(\mathbb{R})$ be a non-negative function and $f\in \mathcal{F}_\lambda$ be a function on an orientable surface $(M,g)$.
Then there exists a constant $\kappa_g>0$ such that
\begin{equation}
\label{mainbound}
\Phi_u(f) \leq \kappa_g(\lambda+1)  \| u\circ f\|.
\end{equation}
\end{thrm}

\smallskip

In order to prove this theorem we compare both sides of inequality (\ref{mainbound}) with an intermediate quantity. Let $\beta(t,f)$ be the number of connected components of $f^{-1}(t)$. Function $\beta(t,f)$ is called the \textit{Banach indicatrix} of $f$ (see \cite{Kron, Yom}). In \cite{PolSod} it was proved that $\int_{-\infty}^{+\infty} u(t)\beta(t,f)dt\leq \kappa_g(\lambda+1)  \| u\circ f\|$ for $f\in \mathcal{F}_\lambda.$ On the other hand, we show that $\Phi_u(f) \leq \int_{-\infty}^{+\infty} u(t)\beta(t,f)dt$, see Proposition \ref{Bound_Topological}.  This proposition, which is of topological nature, constitutes the main technical result of the paper.

\medskip

Now, notice that taking $u\equiv 1$ in (\ref{mainbound}) we get the following corollary:
\begin{cor}\label{cor:ubound}
Let $(M,g)$ be an orientable surface without boundary and let $f \in \mathcal{F}_\lambda$ be a Morse function on $M$. Denote by $l_i$ the lengths of the finite bars of the barcode associated with $f$. Then
\begin{equation}
\label{Exm:Barcode}
\max f - \min f +\sum_i l_i \leq \kappa_g(\lambda+1).
\end{equation}
\end{cor}
\begin{exm}
\label{exampletorus}
The order of $\lambda$ in inequality \eqref{Exm:Barcode} is sharp. Indeed, consider the flat square torus $\mathbb{T}^2=\mathbb{R}^2 /(2\pi \cdot \mathbb{Z})^2.$ We have a sequence $f_n(x,y)=\frac{1}{\pi}\sin (nx)\cos(ny)$, $n\in \mathbb{N}$ of eigenfunctions of $\Delta$  with eigenvalues $2n^2$. By analysing critical points of $f_1=\sin x \cos x$ and using periodicity, one can compute that the full barcode of $f_n$ contains
\begin{itemize}
  \item An infinite bar $(-\frac{1}{\pi},+\infty)$ and $2n^2-1$ copies of finite bar $(-\frac{1}{\pi},0]$ in degree 0;
  \item Two copies of infinite bar $(0,+\infty)$ and $2n^2-1$ copies of finite bar $(0,\frac{1}{\pi}]$ in degree 1;
  \item An infinite bar $(\frac{1}{\pi},+\infty)$ in degree 2.
\end{itemize}
Putting these values in inequality \eqref{Exm:Barcode}  gives us
$$\frac{4}{\pi}n^2 \leq \kappa_g(2n^2+1),$$
which proves that the order of $\lambda$ in \eqref{Exm:Barcode}  is sharp.
\end{exm}
In order to present another application of Theorem \ref{Persistence_Bound} we need the following definition.
\begin{defn}  Let $f:M \rightarrow \mathbb{R}$ be a Morse function on a differentiable manifold $M$ and let $\delta>0$. We say that a critical value $\alpha \in \mathbb{R}$ of the function $f$ is a {\it $\delta$-significant critical value of multiplicity $m$} if the barcode of $f$ contains $m$ bars of length at least $\delta$ having $\alpha$ as one of the endpoints.
\end{defn}
Given $\delta>0$ and a Morse function $f$, let $\mathcal{N}_\delta(f)$ be the number of $\delta$-significant critical values counted with multiplicities.
Theorem \ref{Persistence_Bound} then immediately implies:
\begin{cor}
\label{deltasign}
Let  $(M,g)$ be an orientable surface, possibly with boundary,  and let $f \in \mathcal{F}_\lambda$ be a Morse function on $M$. Then
\begin{equation}
\label{critvalues}
\mathcal{N}_\delta(f) \le \kappa_{g,\delta} (\lambda+1)
\end{equation}
for any $\delta>0$.
\end{cor}
The following example shows that the $\delta$-significance condition for some $\delta>0$ is essential in Corollary \ref{deltasign}. For simplicity, we present it in one dimension, but it could be easily generalized to any dimension.
\begin{exm}
Let $M=\mathbb{S}^1$ be a unit circle and let $N_i$ be {\it any} sequence of natural numbers  tending to infinity.
Consider a  sequence of functions on $M$:
$$f_i(x)=\frac{1}{\sqrt{\pi(2+N_i^{-4})}} \left(1+\frac{1}{N_i^2} \sin(N_i x)\right).$$
It is easy to check that $||f_i||_{L^2(M)}=1$ and $f_i \in \mathcal{F}_\lambda$, $\lambda\ge \frac{1}{\sqrt{2}}$, for all $i=1,2,\dots$. At the same time, the number of critical points, and hence of critical values (counted with multiplicities) is equal to $N_i$, which goes to infinity and hence can not be controlled by $\lambda$.  Note, however,  that for any $\delta>0$,  the number of $\delta$-significant critical values is bounded as $i \to \infty$.
\end{exm}
Estimate \eqref{critvalues} could be also compared to \cite[Theorem 1.1]{Nic}, which shows that the expected value of the number of critical points of a random linear combination of Laplace eigenfunctions $f_1,\dots, f_m$ on a Riemannian manifold satisfies an asymptotic expansion with the leading term of order $m$. Due to Weyl's law, for surfaces this is equivalent to having  the number of critical points of order $\lambda_m$, which agrees with inequality \eqref{critvalues}. Inspired in part by this observation, we propose the following generalization of
\eqref{critvalues} to Riemannian manifolds of arbitrary dimension:
\begin{conj}
\label{higherdimconj_Weak}
Let $(M,g)$ be a Riemannian manifold of dimension $n$, possibly with boundary, and let $f$ be a $L^2$-normalized linear combination of eigenfunctions of $\Delta$ with eigenvalues $\lambda_i \leq \lambda.$ In addition, assume that $f$ is Morse. Then
\begin{equation}
\label{inequality_Weak}
\mathcal{N}_\delta(f) \le \kappa_{g,\delta} (\lambda+1)^{\frac{n}{2}}
\end{equation}
for any $\delta>0$.
\end{conj}

Furthermore, for $n$-dimensional Riemannian manifolds, consider the following generalization of the functional $\Phi_u$: it is defined for Morse functions by an analogue of \eqref{Functional}, the sum being taken over all finite bars in $\mathcal{B}(f)$  in all  degrees. Similarly to \eqref{nonmorsefunct} it also could be extended to arbitrary functions in $C^0(M)$.
\begin{conj}
\label{higherdimconj}
Let  $u\in C(\mathbb{R})$ be a non-negative function and $f$ a $L^2$-normalized linear combination of eigenfunctions of $\Delta$ with eigenvalues $\lambda_i \leq \lambda$ on a Riemannian manifold $(M,g)$.
Then there exists a constant $\kappa_g>0$ such that for any $\lambda>0$,
\begin{equation}\label{inequality_Strong}
{\Phi}_u(f) \leq \kappa_g(\lambda+1)^{\frac{n}{2}} \| u\circ f\|.
\end{equation}
\end{conj}
A possible approach to proving this conjecture is discussed in Remark \ref{Yomdin}.
\begin{exm}\label{rem:variation}
In order to provide intuition about Conjecture \ref{higherdimconj}, let us examine what happens in dimension one (cf. \cite[p. 137]{CSEHM}). In this case, the notions coming from the barcode,  such as the number or the total length of finite bars, have transparent meanings.  Assume that
$(M,g)=(\mathbb{S}^1,g_0)=(\mathbb{R}/(2\pi \cdot \mathbb{Z}),g_0)$ is the circle with the metric inherited from the standard length on $\mathbb{R}$,  and $f:\mathbb{S}^1 \rightarrow \mathbb{R}$ is a Morse function. Since $f$ is Morse, all critical points of $f$ are either local minima or local maxima and they are located on $\mathbb{S}^1$ in an alternating fashion. More precisely, if there are $N$ local minima $x_1,\ldots, x_N$,  there are also $N$ local maxima $y_1,\ldots, y_N$,  and we may label them so that they are cyclically ordered as follows:
$$x_1, y_1, x_2, y_2, \ldots, x_N,  y_N, x_1.$$
Taking $u\equiv 1,$ we have that $\Phi_1(f)=\max f - \min f + \text{the total length of finite bars}.$ All the finite bars appear in degree 0,  and thus by Remark \ref{Critical_Points} we have $N$ finite bars whose left endpoints are $f(x_1),\ldots , f(x_N)$ and whose right endpoints are $f(y_1),\ldots , f(y_N).$ From here it follows that
$$\Phi_1(f)=\sum_{i=1}^{N}(f(y_i)-f(x_i)).$$
On the other hand,  the total variation of $f$ satisfies
$$\text{Var}(f)=2 \sum_{i=1}^{N}(f(y_i)-f(x_i)) = 2\Phi_1(f).$$
Furthermore, using H\"older's inequality and partial integration we have
$$\text{Var}(f)=\int_{0}^{2\pi}|f'(t)|dt\leq \sqrt{2\pi}\bigg( \int_{0}^{2\pi}(f'(t))^2 dt \bigg)^{\frac{1}{2}}=\sqrt{2\pi} \bigg| \int_{0}^{2\pi}f''(t)f(t)dt \bigg|^{\frac{1}{2}}.$$
From Cauchy-Schwarz inequality it follows
$$\text{Var}(f)\leq \sqrt{2\pi} \| f \|^{\frac{1}{2}} \| f'' \|^{\frac{1}{2}}.$$
Finally,  if $f\in \mathcal{F}_\lambda,$ we have $\| f \|^{\frac{1}{2}}=1$ and $\| f'' \|^{\frac{1}{2}} \leq \lambda^{\frac{1}{2}}$ which gives
\begin{equation}
\label{conj1dim}
\frac{1}{2}\text{Var}(f)=\Phi_1(f) \leq \sqrt{ \frac{\pi}{2}} \lambda^{\frac{1}{2}},
\end{equation}
as claimed by  Conjecture \ref{higherdimconj}. In order to extend the result to a  general (not necessarily Morse) $f\in \mathcal{F}_\lambda$,  observe that for every $\epsilon>0$ there exists a sequence of Morse functions $f_n\in \mathcal{F}_{\lambda + \epsilon}$,  such that $d_{C^0}(f,f_n)\rightarrow 0$ when $n\rightarrow \infty.$ For all $k,n\geq 1$ it holds
$$\Phi_{1,k}(f_n)\leq \Phi_1(f_n)\leq \sqrt{\frac{\pi}{2}} (\lambda+\epsilon)^{\frac{1}{2}}.$$
Taking limits for $k,n\rightarrow \infty$ as in (\ref{nonmorsefunct}) and using the fact that $\epsilon>0$ is arbitrary,  we obtain  the inequality  \eqref{conj1dim} for any $f \in \mathcal{F}_\lambda$.
\end{exm}
\begin{exm} The following example shows that the order of $\lambda$ predicted by Conjecture~\ref{higherdimconj} is sharp.
Let   $\mathbb{T}^n=\mathbb{R}^n / (2\pi \cdot \mathbb{Z})^n$ be  the $n$-dimensional torus  equipped with a Euclidean metric $ds^2=\sum dx_i^2$.  Define a sequence of functions
$$f_l(x_1,\ldots , x_n)=\frac{\sqrt{2}}{n(2\pi)^{\frac{n}{2}}}(\sin lx_1 + \ldots + \sin l x_n),~l\in \mathbb{N}.$$
It is easy to check that $\| f_l \|=1$ and $\Delta f_l = l^2 f_l.$ Thus $f_l \in \mathcal{F}_\lambda$ for $\lambda=l^2.$
\end{exm}
\begin{prop}
\label{prop:torus}
There exist constants $A_n$ and $B_n$  such that
$$\Phi_1(f_l)=A_n\lambda^\frac{n}{2}+B_n.$$
\end{prop}
The proof of Proposition \ref{prop:torus} uses the   K\"unneth formula for persistence modules \cite{PSS}, see subsection \ref{subs:torus} for details.

\medskip

Finally, we wish to emphasise that Conjecture \ref{higherdimconj} does not hold for functions in $\mathcal{F}_\lambda$ in dimensions greater than two. This is illustrated by the following example due to Lev Buhovsky \cite{Buh}.

\begin{exm}[Buhovsky's example]\label{example:buh}

For each $n\geq 3,$ we provide a sequence of functions $F_k:\mathbb{T}^n\rightarrow \mathbb{R}$ on $n$-dimensional flat torus $\mathbb{T}^n =\mathbb{R}^n / (2\pi \cdot \mathbb{Z})^n$ such that $\|F_k\|$ and $\| \Delta F_k \|$ are uniformly bounded away from zero and infinity for all $k,$ while $\Phi_1(F_k)$ grows as $k^{n-2}.$ Such sequence violates inequality (\ref{inequality_Strong}).

\medskip

We define $F_k$ as periodic functions on the cube $[-\pi,\pi]^n$ as follows. Let $h:[-1,1]^n\rightarrow [0,1]$ be a bump function. Divide $[-1,1]^n$ into $k^n$ smaller cubes by dividing each interval $[-1,1]$ into $k$ equal parts. Now $h(kx)$ is a bump function supported in $[-\frac{1}{k},\frac{1}{k}]^n$ and we define auxiliary functions $f_k$ to be equal to a copy of $\frac{1}{k^2}h(kx)$ inside each small cube. Since supports of different copies of $\frac{1}{k^2}h(kx)$ are disjoint, $L^2$-orthogonality implies
$$\| f_k\|^2 = k^n \bigg\| \frac{1}{k^2}h(kx) \bigg\|^2 = k^{-4} \| h \|^2,$$
as well as that $\| \Delta f_k \|$ is bounded uniformly in $k.$

Finally, let\footnote{Formally speaking, $F_k$ should be a small perturbation of $f_k+1$ in order to make it Morse, but we will ignore this detail for the sake of clarity.} $F_k=f_k+1.$ This way we obtain a sequence of functions with $\| F_k\|$ and $\| \Delta F_k\|$ bounded away from zero and infinity. At the same time for $t\in (1,1+\frac{1}{k^2})$ the topology of sublevel sets $F_k^{-1}((-\infty,t))$ does not change and each sublevel set is homeomorphic to $\mathbb{T}^n$ with $k^n$ holes. This generates $\sim k^n$ bars of length $\frac{1}{k^2}$ in degree $n-1$ and hence $\Phi_1(F_k)\gg k^{n-2},$ which contradicts (\ref{inequality_Strong}) when $n\geq 3$ because $\frac{F_k}{\|F_k\|}\in \mathcal{F}_\lambda$ with bounded $\lambda,$ but $\Phi_1 \left( \frac{F_k}{\|F_k\|} \right)$ grows as $k^{n-2}.$ A slight modification of this example also yields a counterexample to \eqref{inequality_Weak} in dimensions $n\geq 5.$
\end{exm}

\begin{rem}\label{remark:L^p}
An example similar to Example \ref{example:buh} has been discussed in \cite[Section 5]{CSEHM}. In this paper, $L^p$-versions of functional $\Phi_1$, where the sum is taken over $p$-th powers of the lengths of bars, were considered. The results yield an upper bound for these $L^p$-functionals in terms of the Lipschitz constant of a function. However, for these bounds to hold, it is essential that $p$ is at least the dimension of the base manifold, which can be seen from Example \ref{example:buh}. As a consequence, while the results of \cite{CSEHM} imply some spectral restrictions on the barcodes of Laplace eigenfunctions, they appear to be essentially different from the bounds on $\Phi_1$ obtained in Theorem \ref{Persistence_Bound} and conjectured in Conjecture \ref{higherdimconj}.
\end{rem}

\subsection{Sorting the finite bars of functions in $\mathcal{F}_\lambda$}
\label{sort}
Given a barcode $\mathcal{B}$, write the lengths of its finite bars in the descending order,
$\beta_1 \geq \beta_2 \geq \dots$. The functions $\beta_i(\mathcal{B})$, which are Lipschitz with respect to the bottleneck distance,
are important invariants of barcodes. For instance, $\beta_1$, which was introduced by Usher in \cite{U},  is called {\it the boundary depth}
and has various applications in Morse theory and symplectic topology. The functions $\beta_i$ with $i \geq 2$ are sometimes used
in order to distinguish barcodes, see e.g. \cite{BMMPS}. Fix $\epsilon >0$ and discard all bars of length $< \epsilon$, i.e., introduce
the modified invariant
$$\beta^{(\epsilon)}_i(\mathcal{B}):= \max(\beta_i(\mathcal{B}),\epsilon)\;.$$
\begin{quest}\label{q-1} Assume that the barcode $\mathcal{B}$ contains $N$ finite bars.
What is the optimal (worst-case scenario) running time $T$ of an algorithm which calculates the ordered sequence $\{\beta^{(\epsilon)}_i(\mathcal{B})\}$, $i \geq 1$?
\end{quest}
Since the sharp lower bound on the running time of any comparison sorting algorithm for an array of $N$ real
numbers is ${O}(N\log N)$ (see \cite{CLRS}), the answer to the above question for a general barcode is ${O}(N\log N)$.
Interestingly enough, in some cases Corollary \ref{cor:ubound} enables one to reduce this running time when $\mathcal{B}$ is a barcode of a function from $\mathcal{F}_\lambda$.
More precisely, there exists a constant $c>0$ such that for every $\epsilon > 0$, $\lambda > 0$ and any function $f \in \mathcal{F}_\lambda$ whose barcode contains exactly $N$ finite bars,
one can find a sorting algorithm for all bars from $\mathcal{B}(f)$ of length $\geq \epsilon$ whose running time satisfies
 \begin{equation}\label{eq-runt}
T \leq N + c\cdot\frac{\kappa_g(\lambda+1)}{\epsilon}\cdot\log\frac{\kappa_g(\lambda+1)}{\epsilon}\;.
\end{equation}
Indeed, consider the following algorithm. First compare the length of each bar with $\epsilon$ and
pick only those bars whose length is $\geq \epsilon$. This takes time $N$. Denote by $K$ the number
of chosen bars. Next, perform the optimal sorting algorithm for these  $K$ bars. This takes time ${O}(K\log K)$.
Finally, notice that by Corollary \ref{cor:ubound}, $$K \leq \frac{\kappa_g(\lambda+1)}{\epsilon}\;,$$
which proves \eqref{eq-runt}. In certain regimes, the running time \eqref{eq-runt} is shorter
than the generic bound ${O}(N\log N)$. For instance, if $\lambda$ and $\epsilon$ are fixed and $N \to \infty$, we have
$T = O(N)$.

\medskip

Theorem~\ref{Persistence_Bound} also has applications to questions regarding $C^0$-approximations by functions from  $\mathcal{F}_\lambda$ ,  which is the subject of the next section.


\section{Applications to approximations by eigenfunctions}\label{section-approximations by eigenfunctions}


\subsection{An obstruction to $C^0$-approximations}
As before, let $M$ be an orientable surface, possibly with boundary, endowed with the Riemannian metric $g$, denote by $\text{Diff}(M)$ the set of all diffeomorphisms  of $M$  and assume that $f:M\rightarrow \mathbb{R}$ is a Morse function (vanishing on $\partial M$). We are interested in the question of how well  can $f$ be approximated by functions from $\mathcal{F}_\lambda$ in $C^0$-sense. More precisely, we wish to find a lower bound for the quantity
$$d_{C^0}(f,\mathcal{F}_\lambda)=\inf \{ d_{C^0}(f,h) ~|~ h \in \mathcal{F}_\lambda \},$$
where $d_{C^0}(f,h)=\max\limits_x |f(x)-h(x)|$ as before. In fact, we will study a more general question, namely we will give a lower bound for
$$approx_\lambda(f)=\inf\limits_{\varphi\in \text{Diff}(M)}~d_{C^0}(f\circ \varphi, \mathcal{F}_\lambda).$$
Taking $\varphi = \id_M$ one immediately sees that
$$d_{C^0}(f,\mathcal{F}_\lambda)\geq approx_\lambda(f). $$
We estimate $approx_\lambda(f)$ from below using the information coming from the barcode $\mathcal{B}(f)$. Recall that  the functional $\Phi_1:\mathcal{F}_{\text{Morse}} \rightarrow \mathbb{R} $ defined by (\ref{Functional}) for $u\equiv 1$ gives the sum of the lengths of all finite bars in   $\mathcal{B}(f)$.
{\prop
\label{211}
 For every Morse function $f:M\rightarrow \mathbb{R},$ vanishing on the boundary, the following inequality holds
\begin{equation}\label{Approx_Ineqaulity}
approx_\lambda(f) \geq \begin{cases}
      \frac{1}{2\cdot (|\mathcal{B}'(f)|+1)} \bigg( \Phi_1(f) - \kappa_g(\lambda+1) \bigg) & \text{for}\ \partial M = \emptyset \\
      \frac{1}{2|\mathcal{B}'(f)|+1} \bigg( \Phi_1(f) - \kappa_g(\lambda+1) \bigg) & \text{for}\ \partial M \neq \emptyset \\
    \end{cases}
\end{equation}
\label{Obstruction} }
{\proof From (\ref{LSC_Diffeo}) and (\ref{mainbound}), with $\psi=\id_M$ we obtain
$$ \kappa_g(\lambda+1)\cdot \| u \circ h \| \geq \Phi_u(f)-C(u,f) \cdot d_{C^0}(f\circ \varphi,h),$$
for all Morse $h\in \mathcal{F}_\lambda$ and all diffeomorphisms $\varphi \in \text{Diff}(M),$ with constant $C(u,f)$ being equal to $2\cdot ( |\mathcal{B}'(f)| + 1 )\cdot (\max_{[\min f, \max f]} u )$ or $(2 |\mathcal{B}'(f)| + 1) \cdot (\max_{[\min f, \max f]} u )$ depending on whether $M$ has a boundary. Putting $u\equiv 1$ we have
$$d_{C^0}(f\circ \varphi, h) \geq \frac{1}{2\cdot(|\mathcal{B}'(f)|+1)} \bigg( \Phi_1(f) - \kappa_g(\lambda+1) \bigg),$$
if $\partial M = \emptyset,$ or
$$d_{C^0}(f\circ \varphi, h) \geq \frac{1}{2|\mathcal{B}'(f)|+1} \bigg( \Phi_1(f) - \kappa_g(\lambda+1) \bigg),$$
if $\partial M \neq \emptyset.$ Finally, taking infimum over all $h$ and $\varphi$ and using the fact that Morse functions in $\mathcal{F}_\lambda$ are $C^0$-dense in $\mathcal{F}_\lambda$,  finishes the proof. \qed }
\begin{rem}
The inequality analogous to (\ref{Approx_Ineqaulity}) can be proved  for functions on the  circle $\mathbb{S}^1=\mathbb{R}/(2\pi \cdot \mathbb{Z})$ without referring to the language of barcodes.  Taking into account \eqref{critpoints} and \eqref{conj1dim}, we can restate (\ref{Approx_Ineqaulity}) as
\begin{equation}\label{Approx_Var}
approx_\lambda(f)\geq \frac{1}{|\Crit(f)|} \bigg( \frac{1}{2} \text{Var}(f) -  \sqrt{\frac{\pi}{2}}\lambda^{\frac{1}{2}} \bigg).
\end{equation}
In order to prove (\ref{Approx_Var}) we proceed as in the proof of Proposition \ref{Obstruction}. One readily checks that
\begin{equation}
\label{readily}
\frac{1}{2}\text{Var}(f) - \frac{1}{2}\text{Var}(h) \leq |\Crit(f)|d_{C^0}(f,h).
\end{equation}
Indeed, as in Example \ref{rem:variation},  if $x_1,\ldots , x_N$ are local minima and $y_1,\ldots,y_N$ are local maxima of $f$,  we have that
$$\frac{1}{2}\text{Var}(f)=\sum_{i=1}^N(f(y_i)-f(x_i)).$$
On the other hand,
$$\frac{1}{2}\text{Var}(h)\geq \sum_{i=1}^N(h(y_i)-h(x_i)).$$
Subtracting the latter expression from the former   yields (\ref{readily}), which together with (\ref{conj1dim}) gives
$$\frac{1}{|\Crit(f)|} \bigg( \frac{1}{2} \text{Var}(f) -  \sqrt{\frac{\pi}{2}}\lambda^{\frac{1}{2}} \bigg) \leq d_{C^0}(f,\mathcal{F}_\lambda).$$
Since $|\Crit(f)|$ and $\text{Var}(f)$ do not change when $f$ is composed with a diffeomorphism,  (\ref{Approx_Var}) follows.
\end{rem}
\begin{rem}
\label{nonmorserem}
The following analogue of Proposition \ref{Obstruction} holds for any function $f\in C^0(M)$. For any $k=1,2,\dots$, we have
\begin{equation}
approx_\lambda(f) \geq \begin{cases}
      \frac{1}{2k+2} \bigg( \Phi_{1,k}(f) - \kappa_g(\lambda+1) \bigg) & \text{for}\ \partial M = \emptyset \\
      \frac{1}{2k+1} \bigg( \Phi_{1,k}(f) - \kappa_g(\lambda+1) \bigg) & \text{for}\ \partial M \neq \emptyset \\
    \end{cases}
\end{equation}
The proof is the same, with the constant $C(u,f)$ replaced by the Lipschitz constant from  Lemma \ref{Lipsch}.
\end{rem}
\begin{cor}
\label{214}
Let $M$ be a surface without boundary and $f:M \rightarrow \mathbb{R}$ be a Morse function.
Suppose that $approx_\lambda(f)\le \epsilon$ for some $\epsilon>0$, and the barcode $\mathcal{B}(f)$ contains $N$ finite bars  of length at least $L+2 \varepsilon$ each, for some $L>0$. Then
\begin{equation}\label{Eigenvalue_Estimate}
\lambda \geq \frac{1}{\kappa_g}(N+1)L-1.
\end{equation}
\end{cor}
\begin{proof}
Indeed, it follows from the assumptions on the barcode of $f$ that  $\Phi_{1,N}(f)\geq (N+1)(L+2\varepsilon)$, which together with Remark~\ref{nonmorserem} yields
$$\varepsilon \geq \frac{1}{2(N+1)} \bigg( (N+1)(L+2\varepsilon) - \kappa_g(\lambda+1)   \bigg).$$
Rearranging this inequality we obtain \eqref{Eigenvalue_Estimate}.
\end{proof}
{\rem From (\ref{Eigenvalue_Estimate}) we see how $\lambda$, which is needed to uniformly approximate $f$ by functions from $\mathcal{F}_\lambda$,  grows with $N$ and $L$. Informally speaking, one may think of $N$ as  a measure of how much $f$ oscillates, while $L$ gives a lower bound on the amplitude of these oscilations. The above inequality should then be understood as a quantitative version of the informal principle that the more the function oscillates and the bigger the oscillations, the larger eigenvalues of the Laplacian are needed to approximate it with a normalized linear combination of the corresponding eigenfunctions. We refer to \cite{W2} for other applications of persistence to approximation theory.}

\subsection{Modulus of continuity and average length of bars on $\mathbb{T}^2$}
 Proposition~\ref{Obstruction} gives an obstruction to approximating functions by functions from $\mathcal{F}_\lambda.$ As we mentioned before, $\mathcal{F}_\lambda$ contains normalized linear combinations of eigenfunctions of $\Delta$ with eigenvalues not greater than $\lambda.$ In the case of flat torus $\mathbb{T}^2=\mathbb{R}^2 / (2\pi \cdot \mathbb{Z})^2$ these eigenfunctions are trigonometric polynomials and Proposition~\ref{Obstruction} may be interpreted as an inverse statement about $C^0$-approximations by trigonometric polynomials. A direct theorem about $C^0$-approximations by trigonometric polynomials  on $n$-dimensional flat torus was proved in \cite{Judin} (theorems of this type are sometimes referred to as Jackson's theorems, see \cite{Pinkus} for a survey), consequently giving an upper bound for $approx_\lambda (f)$ in terms of moduli of continuity and smoothness of $f.$ We combine this result with Proposition~\ref{Obstruction} to obtain a relation between the average length of a bar in a barcode of a  Morse function on $\mathbb{T}^2$ and its modulus of continuity which is defined below.
\\

Assume $M=\mathbb{T}^2=\mathbb{R}^2 / (2\pi \cdot \mathbb{Z})^2$, let $f:M\rightarrow \mathbb{R}$ be a continuous function, $\delta >0$ a real number and denote by
$$\omega_1(f,\delta)=\sup_{|t|\leq \delta} \max_x |f(x+t)-f(x)|,$$
the {\it modulus of continuity} of $f$ and by
$$\omega_2(f,\delta)=\sup_{|t|\leq \delta} \max_x |f(x-t)-2f(x)+ f(x+t)|,$$
the {\it modulus of smoothness} of $f.$ One readily checks that
\begin{equation}\label{Moduli}
  \omega_2(f,\delta) \leq 2 \omega_1(f,\delta).
\end{equation}
Let
$$\mathcal{T}_\lambda= \bigg\langle \sin(v_1 x+v_2 y), \cos(v_1 x+v_2 y) ~ \bigg| ~ v_1^2+v_2^2 \leq \lambda \bigg\rangle_{\mathbb{R}},$$
be the space of trigonometric polynomials on $M$ whose eigenvalues (as eigenfunctions of $\Delta$) are bounded by $\lambda.$
The following porposition was proved in \cite{Judin}:
\begin{prop}
 For every continuous function $f:M\rightarrow \mathbb{R}$ it holds
\begin{equation}\label{Modulus_Smoothness}
  d_{C^0}(f,\mathcal{T}_\lambda) \leq 2 \omega_2 \bigg(f,\frac{C_0}{\sqrt{\lambda}} \bigg),
\end{equation}
where $C_0>0$ is a constant.
\label{Yudin}
\end{prop}
By (\ref{Moduli}) we also have that
\begin{equation}\label{Modulus_Continuity}
  d_{C^0}(f,\mathcal{T}_\lambda) \leq 4 \omega_1 \bigg(f,\frac{C_0}{\sqrt{\lambda}} \bigg).
\end{equation}
{\rem Constant $C_0$ is computed in \cite{Judin} to be $C_0=\sqrt{\mu_1(D^2(\frac{1}{2}))}$, where $\mu_1(D^2(\frac{1}{2}))$ is the first Dirichlet eigenvalue of $\Delta$ inside the 2-dimensional disk $D^2(\frac{1}{2})$ of radius $\frac{1}{2}$.}

Our goal is to prove the following result which shows that the average bar length of a Morse function $f$ on a flat torus $M$ could be {\it uniformly} controlled by the  $L^2$-norm of $f$ and the  modulus of continuity of $f$ on the scale $1/\sqrt{|\mathcal{B'}(f)|}$.
\begin{thrm}
\label{modulcont}
There exist constants $C_0, C_1>0$ such that for any $f \in \mathcal{F}_{\text{Morse}}$ on a flat torus $M=\mathbb{T}^2$,
\begin{equation}
\label{averagebar}
\frac{\Phi_1(f)}{|\mathcal{B}'(f)|+1} \le C_1\|f\| + 8\omega_1\left(f,\frac{C_0}{\sqrt{|\mathcal{B}'(f)|}}\right).
\end{equation}
\end{thrm}
\begin{proof}
Inspecting the proof of Proposition~\ref{Yudin} in \cite{Judin}, we observe that it relies on an explicit construction of a function $h,$ depending on $f$ and $\lambda,$ which satisfies
\begin{equation}\label{Approximation_Function}
  d_{C^0}(f,h)\leq 2 \omega_2 \bigg(f,\frac{C_0}{\sqrt{\lambda}} \bigg).
\end{equation}
Our goal is to estimate $d_{C^0}(f,h)$ from below using Proposition~\ref{Obstruction}. However, a priori we do not have any information about the $L^2$-norm of $h$ and Proposition~\ref{Obstruction} relates to distance from functions in $\mathcal{F}_\lambda$ whose $L^2$-norm is equal to one. In order to overcome this issue, we present the construction of the approximation-function $h$ and prove that $\| h \| \leq \| f \|.$
\\
\\
For a vector $v\in \mathbb{Z}^2$ let
$$c_v(f)=\frac{1}{(2\pi)^2}\int_{\mathbb{T}^2} f(x) e^{-i \langle v, x \rangle} dx,$$
be the corresponding Fourier coefficient of $f.$ Take $U$ to be the first Dirichlet eigenfunction of $\Delta$ inside the disk $D^2(\frac{1}{2})$ of radius $\frac{1}{2}$, normalized by $\| U \|_{L^2(D^2(\frac{1}{2}))}=1,$ and $V$ its extension by zero to the whole plane, i.e.
$$ V(x)=\begin{cases}
          U(x), & \mbox{if } x\in D^2(\frac{1}{2})  \\
          0, & \mbox{otherwise}.
        \end{cases} $$
If we denote by $W=V*V$ the convolution of $V$ with itself, the desired approximation is given by the formula
\begin{equation}\label{Korovkin_Mean}
  h(x)=\sum_{ \substack{ v \in \mathbb{Z}^2 \\ |v| \leq \sqrt{\lambda}} } c_v(f)\cdot W\bigg(\frac{v}{\sqrt{\lambda}} \bigg)e^{i\langle v,x \rangle},
\end{equation}
where $|v|$ stands for the standard Euclidean norm on $\mathbb{R}^2.$ The function $h$ defined by (\ref{Korovkin_Mean}) is called the \textit{multidimensional Korovkin's mean}. It satisfies (\ref{Approximation_Function}), as proved in \cite{Judin}, and since $\sum_{v\in \mathbb{Z}^2} c_v(f)e^{i\langle v,x \rangle}$ is the Fourier expansion of $f$, we have that
$$\| h \| \leq (\max_{x\in\mathbb{R}^2} |W(x)|)\cdot \| f \|.$$
By using the Cauchy-Schwarz inequality we obtain
$$|W(x)|\leq \int_{\mathbb{R}^2}|V(t)|\cdot |V(x-t)|dt \leq \sqrt{\int_{\mathbb{R}^2}|V(t)|^2dt } \cdot \sqrt{\int_{\mathbb{R}^2}|V(x-t)|^2dt}=\| V \|^2=1,$$
which yields $\| h \| \leq \| f \|.$
\\
\\
We now proceed with analysing (\ref{Approximation_Function}). First note that
$$ \| h \| \cdot d_{C^0} \bigg( \frac{f}{\|h\|} , \mathcal{F}_\lambda \bigg) \leq \| h \| \cdot d_{C^0} \bigg(\frac{f}{\| h \|},\frac{h}{\| h \|} \bigg) \leq 2 \omega_2 \bigg(f,\frac{C_0}{\sqrt{\lambda}} \bigg),$$
because $\frac{h}{\| h \|} \in \mathcal{F}_\lambda.$ The last inequality together with Proposition~\ref{Obstruction} gives
\begin{equation}\label{Equation_Long}
\frac{\| h \|}{2\cdot(|\mathcal{B'}(\frac{f}{\| h \|} )|+1)} \bigg( \Phi_1 \bigg(\frac{f}{\| h \|} \bigg) - \kappa_0(\lambda+1) \bigg)  \leq 2 \omega_2 \bigg(f,\frac{C_0}{\sqrt{\lambda}} \bigg).
\end{equation}
Here $\kappa_0=\kappa_g$ for $g$ being the flat metric on $M$.
Multiplying the function by a positive constant results in multiplying the endpoints of each bar in the barcode by the same constant. Thus, the total number of bars does not change after multiplication, while the lengths of finite bars scale with the same constant. In other words, we have that $\bigg | \mathcal{B'}\bigg( \frac{f}{\| h \|} \bigg) \bigg|= | \mathcal{B'}(f)|$ and $\| h \| \cdot \Phi_1 \bigg( \frac{f}{\| h \|} \bigg)=\Phi_1(f).$ Substituting these equalities in (\ref{Equation_Long}) and using $\| h \| \leq \| f \|$ we obtain
\begin{equation}
\label{smooth2}
\frac{1}{4\cdot(|\mathcal{B'}(f)|+1)} \bigg( \Phi_1(f) - \kappa_0(\lambda+1) \| f \| \bigg) \leq \omega_2 \bigg(f,\frac{C_0}{\sqrt{\lambda}} \bigg),
\end{equation}
and by (\ref{Moduli}) also
$$\frac{1}{8\cdot(|\mathcal{B'}(f)|+1)} \bigg( \Phi_1(f) - \kappa_0(\lambda+1) \| f \| \bigg) \leq \omega_1 \bigg(f,\frac{C_0}{\sqrt{\lambda}} \bigg).$$
Setting $\lambda=|\mathcal{B}'(f)|$ and $C_1=\kappa_0$ in the last inequality completes the proof of the theorem.
\end{proof}
\begin{rem}
As follows from  Remark \ref{nonmorserem} and the proof of Theorem \ref{modulcont} above, for any $k\ge 1$ and any $f\in C^0(M)$ we have:
\begin{equation}
\label{averagebark}
\frac{\Phi_{1,k}(f)}{k+1} \le C_1\|f\| + 8 \omega_1\left(f,\frac{C_0}{\sqrt{k}}\right).
\end{equation}
The left-hand side of \eqref{averagebark} could be interpreted as  the average length of a  bar among the $k$ longest bars in the barcode of $f$.
\end{rem}
\begin{rem}\label{Remark:Ganzburg}
Note that formula \eqref{smooth2} implies
\begin{equation}
\label{ganz}
\frac{\Phi_1(f)}{|\mathcal{B}'(f)|+1} \le C_1\|f\| + 4\omega_2\left(f,\frac{C_0}{\sqrt{|\mathcal{B}'(f)|}}\right).
\end{equation}
In fact, Theorem \ref{modulcont} admits the following generalization. Given a smooth function $f$ on a flat torus $M=\mathbb{T}^2$, define its {\it modulus of smoothness of order $m$} by
$$
\omega_m(f,\delta)=\sup_{|t|\le \delta}\max_x \left|\sum_{j=0}^m (-1)^{(m-j)} \binom{m}{j} f(x+jt) \right|.
$$
From the results of \cite{Ganzburg}, it can be deduced that
$$d_{C^0}(f,\mathcal{T}_\lambda) \leq C_2(k)\, \omega_{2k} \bigg(f,\frac{C_0(k)}{\sqrt{\lambda}} \bigg)$$
for some positive constants $C_0(k), C_2(k)$ which depend on $k.$ Similarly to the proof of Theorem \ref{modulcont}, one then obtains
$$
\frac{\Phi_1(f)}{|\mathcal{B}'(f)|+1} \le C_1(k)\|f\| + 2C_2(k)\, \omega_{2k}\left(f,\frac{C_0(k)}{\sqrt{|\mathcal{B}'(f)|}}\right).
$$
Constants $C_0(k), C_1(k), C_2(k)$ could be computed  explicitly.
\end{rem}
\begin{exm} The following example shows that the choice of the scale  in the modulus of continuity on the right hand side of \eqref{averagebark}   is optimal.
Take a unit disk $B_1$ inside the  torus $M$ and let $\chi$ be a smooth cut-off function supported in $B$ and equal to one in $B_{\frac{1}{2}}$. Let
$g_n(x,y)=\chi(x,y) \sin nx \cos ny$, $n\in \mathbb{N}$ be a sequence of functions on the torus. For any $0<s<1$, set $g_{n,s}(x,y)=g(\frac{x}{s},\frac{y}{s})$. Let $\alpha \ge 1$ be some real number.
Choose $k=n^2$ and $s=n^{\frac{1-\alpha}{2}}$. It suffices to verify that  the inequality
\begin{equation}
\label{optimalexample}
\frac{\Phi_{1,n^2}(g_{n,s})}{n^2+1} \le C_1\|g_{n,s}\| + 8 \omega_1\left(g_{n,s},\frac{C_0}{n^{\alpha}}\right).
\end{equation}
holds for all $n$ only for $\alpha=1$. Indeed, take any $\alpha>1$.  Note that the  left-hand side of \eqref{optimalexample} is bounded away from zero as $n\to \infty$, since the number of bars of  unit length in the barcode of $g_{n,s}((x,y))$
is of order $n^2$. At the same time,  $\|g_{n,s}\| =s\|g_n\| \to 0$ as $n\to\infty$, since $s=n^{\frac{1-\alpha}{2}} \to 0$. Moreover, estimating the derivatives of $g_{n,s}$ one can easily check that
$$
 \omega_1\left(g_{n,s},\frac{C_0}{n^{\alpha}}\right) = O(n^{-\alpha})\, O(n\cdot n^{\frac{\alpha-1}{2}})=O\left(n^{\frac{1-\alpha}{2}}\right)=o(1)
$$
for any $\alpha>1$.
Therefore, inequality \eqref{optimalexample} is violated for $\alpha>1$ for $n$ large, and hence the choice $\alpha=1$ is optimal.

Note that, while the functions $g_{n,s}(x,y)$ are compactly supported and hence not Morse, they could be made Morse by adding a small perturbation.  A similar argument would then yield optimality of the  $1/\sqrt{|\mathcal{B'}(f)|}$ scale in the modulus of continuity on the right-hand side of inequality \eqref{averagebar}.
\end{exm}

\begin{rem} It would be interesting to generalize Theorem~\ref{modulcont} to an arbitrary Riemannian surface. In order to do that we need an analogue of Proposition~\ref{Yudin}. For a different version of Jackson type approximation theorems on Riemannian manifolds see \cite[Lemma 4.1]{IP} and \cite[Lemma 9.1]{FFP}.
\end{rem}

\section{Barcodes and the Banach indicatrix}
\subsection{A topological bound on the Banach indicatrix}
We proceed with some general topological considerations. Let $M$ be a Riemannian manifold and assume that $f|_{\partial M}=0$ if $\partial M \neq \emptyset$, $0$ being a regular value.  Let $t\neq 0$ be another regular value of $f$ and denote by $M^t=f^{-1}((-\infty,t])$. $M^t\subset M$ is a submanifold with boundary $\partial M^t=f^{-1}(t)$ for $t<0$ or $\partial M^t=f^{-1}(t) \sqcup \partial M$ for $t>0.$ Recall that Banach indicatrix $\beta(t,f)$ denotes the number of connected components of $f^{-1}(t).$ By description of $\partial M^t,$ one sees that $\beta(t,f)$ essentially counts the number of the boundary components of $M^t$ (up to the boundary components of the whole manifold $M$). We will exploit this fact to estimate $\beta(t,f)$ from below using information coming from barcode $\mathcal{B}(f).$ If we denote by $\chi_I$ the characteristic function of the interval $I$, the following proposition holds.
{\prop Let $f \in \mathcal{F}_{Morse}$ on a Riemannian manifold $M$ of dimension $n$. Denote by $(x_i^{(k)},y_i^{(k)}]\in \mathcal{B}_k(f)$ the finite bars in the degree $k$ barcode of $f$ and let $t\neq 0$ be a regular value. If $\partial M = \emptyset$ it holds
\begin{equation}\label{Estimate_Top_NoBdry}
\chi_{(\min f,\max f]}(t) + \sum_i \chi_{(x_i^{(0)},y_i^{(0)}]}(t) + \sum_j \chi_{(x_j^{(n-1)},y_j^{(n-1)}]}(t) \leq \beta(t,f),
\end{equation}
\noindent and if $\partial M \neq \emptyset$ it holds
\begin{equation}\label{Estimate_Top_Bdry}
\chi_{(\min f,0]}(t) + \sum_i \chi_{(x_i^{(0)},y_i^{(0)}]}(t) + \sum_j \chi_{(x_j^{(n-1)},y_j^{(n-1)}]}(t) \leq \beta(t,f).
\end{equation}
 \label{Bound_Topological} }

We defer proving Proposition~\ref{Bound_Topological} and first deduce Theorem~\ref{Persistence_Bound} using it.

\begin{rem}
One may easily check that if $M=S^2$ the inequality (\ref{Estimate_Top_NoBdry}) becomes an equality.
\end{rem}

\begin{rem} \label{rem-reebgraph} If $\dim M=2$, integrating inequalities in Proposition~\ref{Bound_Topological} gives an upper bound on the total length of the finite bars in the barcode of a function $f$ in terms of the integral of its Banach indicatrix. The latter quantity admits an interpretation as the total length of the Reeb graph of a function $f$ with respect to a natural metric incorporating the oscillations
of $f$. It is likely that an analogue of the functional $\Phi_u$ defined in this setting is robust with respect to the distance on Reeb graphs introduced in \cite{BGW, BMW}, and that this way one could get applications to approximation theory similar to the ones obtained in Section~\ref{section-approximations by eigenfunctions}. We plan to explore this route elsewhere.
\end{rem}

\subsection{Proof of Theorem \ref{Persistence_Bound}}
Let us now restrict to the two-dimensional case and assume that $M$ is an orientable surface, possibly with boundary, equipped with Riemannian metric $g$.
First note that it suffices to verify inequality \eqref{mainbound} for Morse functions. Indeed, suppose that the inequality is proved for Morse functions in $\mathcal{F}_\lambda$ for all
$\lambda>0$ and let  $f\in \mathcal{F}_\lambda$ be arbitrary. One can easily check that for any $\epsilon>0$ there exists a sequence of Morse functions
$f_n \in \mathcal{F}_{\lambda+\epsilon}$ such that $d_{C^0}(f,f_n)\to 0$ as $n\to \infty$. For all $k,n\ge 1$ we have
$$\Phi_{u,k}(f_n) \le\Phi_u(f_n) \le \kappa_g (\lambda+1+\epsilon)\|u\circ f_n \|,$$
where the first inequality follows from the definition \eqref{Functionalk} of the functional $\Phi_{u,k}$ and the second inequality holds by the assumption that \eqref{mainbound} is true for Morse functions. Taking the limits as $k$ and $n$ go to infinity in definition \eqref{nonmorsefunct} and using the fact that $\epsilon>0$ is arbitrary,  we obtain that \eqref{mainbound} holds
for the function $f$.

It remains to prove inequality \eqref{mainbound} when $f\in \mathcal{F}_\lambda$ is Morse.
Denote by $\| \cdot \|$ the $L^2$-norm with respect to Riemannian area $\sigma$ and by $\Delta$ the Laplace-Beltrami operator with respect to $g$. The analytical tool that we are going to use is   \cite[Theorem 1.5]{PolSod} which gives us that for any continuous function $u\in C(\mathbb{R})$ and any smooth function $f$ on $M$ (which is assumed to be equal to zero on the boundary if $\partial M \neq \emptyset$), the following inequality holds
\begin{equation}\label{Generalized_Indicatrix}
\int\limits_{-\infty}^{+\infty}u(t)\beta(t,f)~dt \leq \kappa_g ( \| f \| + \| \Delta f \| )\cdot \| u\circ f \|,
\end{equation}
where $\kappa_g$ depends on the Riemannian metric $g$.
\\

If we assume that $f\in \mathcal{F}_\lambda$ in (\ref{Generalized_Indicatrix}), we immediately get
\begin{equation}\label{Inequality_Indicatrix}
\int\limits_{-\infty}^{+\infty}u(t)\beta(t,f)~dt \leq \kappa_g ( \lambda +1 )\cdot \| u\circ f \|.
\end{equation}
Since $f$ is Morse we can apply Proposition~\ref{Bound_Topological}. Combining this proposition with inequality~(\ref{Inequality_Indicatrix}) immediately yields
Theorem \ref{Persistence_Bound}. \qed
\begin{rem} It follows from the proof of Theorem \ref{Persistence_Bound} that inequality \eqref{mainbound} holds for any function in  the closure of
$\mathcal{F}_\lambda$ in $C^0$ topology.
\end{rem}
\begin{rem}
\label{Yomdin}
The proof of Theorem \ref{Persistence_Bound} suggests the following approach to proving Conjecture \ref{higherdimconj}. Recall that, by definition,  the Banach indicatrix is given by $\beta(t,f)~=~b_0(f^{-1}(t))$.
In view of Proposition \ref{Bound_Topological}, it is plausible that the following inequality holds in dimension $n\ge 3$:
$$\Phi_u(f) \le \int_{-\infty}^{+\infty}\left(\sum_{i=0}^{n-2} b_i(f^{-1}(t))\right) u(t) dt,$$
where $b_i$ is the $i$-th Betti number. As follows from  \cite{Yom}, the quantity on the right-hand side could be  bounded from above  using  the uniform  norm of the derivatives of $f$ (see also \cite{LL} for related recent developments). In order to establish Conjecture \ref{higherdimconj},  one would need to prove  a higher-dimensional analogue of \cite[Theorem 1.5]{PolSod}, allowing to replace the uniform estimates by $L^2$-bounds.
\end{rem}

\subsection{Proof of Proposition~\ref{Bound_Topological}}
Recall the notation previously introduced. For $t$ regular value of $f$, we denote $M^t=f^{-1}((-\infty,t]) \subset M$. As mentioned before $M^t$ is a submanifold with boundary $\partial M^t=f^{-1}(t)$ for $t<0$ or $\partial M^t=f^{-1}(t) \sqcup \partial M$ for $t>0.$ Let $b_k(t)=\dim H_k(M^t; \mathbb{R})$ be the Betti numbers of $M^t.$ We assume that $\min f<t<\max f$, since otherwise the inequalities obviously hold (both sides are equal to zero). We will always work with homologies with coefficients in $\mathbb{R}$ and will omit the coefficients from the notation.
\\
\\
Let $j:f^{-1}(t)\rightarrow M^t$ be the inclusion and denote by $j_*:H_{n-1}(f^{-1}(t))\rightarrow H_{n-1}(M^t)$ the induced map in homology. Since we work over $\mathbb{R}$, $j_*(H_{n-1}(f^{-1}(t))) \subset H_{n-1}(M^t)$ is a vector subspace. First, we claim that
\begin{equation}\label{Indicatrix_1}
 \beta(t,f)= \dim H_n (M^t,f^{-1}(t)) + \dim (j_*(H_{n-1}(f^{-1}(t)))).
\end{equation}
To prove this, we examine the following part of the long exact sequence of the pair $(M^t,f^{-1}(t))$:
$$H_n(M^t)\rightarrow H_n(M^t, f^{-1}(t)) \xrightarrow{\partial} H_{n-1}(f^{-1}(t)) \xrightarrow{j_*} H_{n-1}(M^t). $$
Since $f^{-1}(t)$ is an $(n-1)$-dimensional  orientable manifold, we have $\beta(t,f)=\dim H_{n-1}(f^{-1}(t))$,  and by the Rank-nullity theorem
$$\beta(t,f)=\dim H_{n-1}(f^{-1}(t))=  \dim (\ker j_*) + \dim (j_*(H_{n-1}(f^{-1}(t)))).$$
By the exactness $\dim (\ker j_*)=\dim (\im \partial)$ and, because $H_n(M^t)=0$, $\partial$ is an inclusion, which means that $\dim (\im \partial)=\dim H_n(M^t,f^{-1}(t))$ and (\ref{Indicatrix_1}) follows.
\\
\\
Second, we note that
\begin{equation}\label{Finite_Bars_Degree_1}
\sum_j \chi_{(x_j^{(n-1)},y_j^{(n-1)}]}(t)= \dim (\ker i_*),
\end{equation}
where $i:M^t \rightarrow M$ is the inclusion and $i_*:H_{n-1}(M^t) \rightarrow H_{n-1}(M)$ induced map on homology. This comes from the fact that finite bars in barcode of $f$ correspond to homology classes which appear throughout filtration process, but do not exist in actual homology of $M.$ Denote by $\widetilde{M}^t=f^{-1}([t,+\infty))$, $\widetilde{M}^t \cap M^t=f^{-1}(t)$, and by $\tilde{j}:f^{-1}(t) \rightarrow \widetilde{M}^t$ and $\tilde{i}:\widetilde{M}^t \rightarrow M$ the inclusions. We examine the following part of the Mayer-Vietoris sequence
$$H_{n-1}(f^{-1}(t)) \xrightarrow{(j_*,\tilde{j}_*)} H_{n-1}(M^t)\oplus H_{n-1}(\widetilde{M}^t)\xrightarrow{i_*-\tilde{i}_*} H_{n-1}(M). $$
From the exactness we have
$$\ker i_* \cong \ker (i_* - \tilde{i}_*) \cap (H_{n-1}(M^t),0)=\im (j_*,\tilde{j}_*) \cap (H_{n-1}(M^t),0),$$
while on the other hand
$$\im (j_*, \tilde{j}_*)\cap (H_{n-1}(M^t),0)= \{ (j_* (a), \tilde{j}_* (a)) | a\in H_{n-1}(f^{-1}(t)),\tilde{j}_*(a)=0 \}=$$
$$= \{ (j_* (a), 0) | a\in \ker \tilde{j}_* \subset H_{n-1}(f^{-1}(t)) \} \cong j_*(\ker \tilde{j}_*),$$
and thus
$$\ker i_* \cong j_*(\ker \tilde{j}_*).$$
However, since $\tilde{i}_* \circ \tilde{j}_* = i_* \circ j_*,$ we see that $j_*(\ker \tilde{j}_*) \subset \ker i_*$ and thus
\begin{equation}\label{Kernel_Formula}
  \ker i_* = j_*(\ker \tilde{j}_*).
\end{equation}
From now on, we distinguish two cases.
\\
\\
\textbf{Case I - $\partial M^t=f^{-1}(t)$}
\\
\\
This case covers the situation when $\partial M=\emptyset$ and when $\partial M \neq \emptyset$, but $t<0$. The left hand sides of the inequalities (\ref{Estimate_Top_NoBdry}) and (\ref{Estimate_Top_Bdry}) are equal for $t<0$ and thus we need to prove (\ref{Estimate_Top_NoBdry}). By the case-assumption, we have that
$$\dim H_n(M^t, f^{-1}(t))=b_0(t),$$
and from the definition of barcode we know that
$$b_0(t)=\chi_{(\min f,\max f]}(t)+\sum_i \chi_{(x_i^{(0)},y_i^{(0)}]}(t).$$
Combining these equalities with (\ref{Indicatrix_1}) renders the statement into
$$\sum_j \chi_{(x_j^{(n-1)},y_j^{(n-1)}]}(t) \leq \dim (j_*(H_{n-1}(f^{-1}(t)))),$$
which after substituting (\ref{Finite_Bars_Degree_1}) and (\ref{Kernel_Formula}) becomes
$$\dim (j_*(\ker \tilde{j}_*)) \leq \dim (j_*(H_{n-1}(f^{-1}(t)))).$$
This inequality is obvious because $\ker \tilde{j}_* \subset H_{n-1}(f^{-1}(t))$.
\\
\\
\textbf{Case II - $\partial M^t=f^{-1}(t) \sqcup \partial M,~\partial M \neq \emptyset$}
\\
\\
This case covers the situation when $\partial M\neq \emptyset$ and $t>0.$ We need to prove (\ref{Estimate_Top_Bdry}), which for $t>0$ becomes
$$\sum_i \chi_{(x_i^{(0)},y_i^{(0)}]}(t) + \sum_j \chi_{(x_j^{(n-1)},y_j^{(n-1)}]}(t) \leq \beta(t,f).$$
Denote by $\partial M = \Sigma_1 \sqcup \ldots \sqcup \Sigma_l$ the boundary components of the whole manifold $M$, where $\Sigma_i$ are connected, orientable, $(n-1)$-dimensional manifolds. Now the boundary of $M^t$ is $\partial M^t=f^{-1}(t) \sqcup \Sigma_1 \sqcup \ldots \sqcup \Sigma_l.$ We may divide connected components of $M^t$ into two groups, one of which consists of all the components whose boundary lies entirely in $f^{-1}(t)$ and the other one consists of all the components whose boundary contains at least one $\Sigma_i$ (i.e. the boundary of these components is a mix of parts of $\partial M$ and $f^{-1}(t)$). Denote by $k$ the number of connected components of $M^t$ whose boundary contains at least one $\Sigma_i.$ Now, since
$$\dim H_n(M^t, f^{-1}(t) \sqcup \partial M)=b_0(t),$$
we have that
$$\dim H_n(M^t, f^{-1}(t))=b_0(t)-k,$$
and thus by (\ref{Indicatrix_1})
$$\beta(t,f)=b_0(t)-1 + \dim (j_*(H_{n-1}(f^{-1}(t))))-(k-1).$$
Since $M$ is connected
$$b_0(t)-1=\sum_i \chi_{(x_i^{(0)},y_i^{(0)}]}(t),$$
and hence we need to prove that
$$\sum_j \chi_{(x_j^{(n-1)},y_j^{(n-1)}]}(t) \leq \dim (j_*(H_{n-1}(f^{-1}(t))))-(k-1).$$
Using (\ref{Finite_Bars_Degree_1}) we transform the statement into
$$\dim (\ker i_*) + k-1 \leq \dim (j_*(H_{n-1}(f^{-1}(t)))).$$
In order to prove this inequality, we will find $k-1$ linearly independent vectors in the quotient space $j_*(H_{n-1}(f^{-1}(t))) / \ker i_*$ (note that by (\ref{Kernel_Formula}) we have that $\ker i_*=j_*(\ker \tilde{j}_*) \subset j_*(H_{n-1}(f^{-1}(t)))$). Assume that $k\geq 2$ (because otherwise the statement is trivial) and denote by $M_1^t, \ldots , M_k^t$ the connected components of $M^t$ whose boundary contains some $\Sigma_i.$ We know that for $1\leq i \leq k,$ homology class $0=[\partial M_i^t]\in H_{n-1}(M^t)$ decomposes as $0=[\partial M_i^t]=d_i+e_i,$ where $d_i=[\Sigma_{i_1}]+\ldots + [\Sigma_{i_{m_i}}]$ for some $[\Sigma_{i_1}], \ldots , [\Sigma_{i_{m_i}}]$ and $e_i\in j_*(H_{n-1}(f^{-1}(t))).$
Moreover, since $M_i^t$ are disjoint, we have that
$$d_1+\ldots + d_k = [\Sigma_1]+\ldots + [\Sigma_l],$$
and $d_1, \dots , d_k$ partition the set $\{[\Sigma_1], \ldots, [\Sigma_l] \}.$ We have that $$d_1,\ldots , d_k \in j_*(H_{n-1}(f^{-1}(t))),$$  because $d_i=-e_i$, and let $[d_1],\ldots , [d_k] \in j_*(H_{n-1}(f^{-1}(t))) / \ker i_*$ be the corresponding classes inside the quotient space. We claim that any $k-1$ of $[d_1],\ldots , [d_k]$ are linearly independent. Once we prove this, choosing and $k-1$ of these gives us the $k-1$ classes that we need.
\\
\\
First, we observe that $[\Sigma_1],\ldots , [\Sigma_l] \in H_{n-1}(M^t)$ are linearly independent. In order to prove this,  consider the following part of the long exact sequence of the pair
$(M^t,\partial M)$:
$$H_n(M^t,\partial M) \rightarrow H_{n-1}(\partial M) \rightarrow H_{n-1}(M^t).$$
Note that $H_n(M^t,\partial M)=0$.  Indeed, if  $H_n(M^t,\partial M)\neq 0$,  then $M^t$ contains a connected component $N$, such that $\partial N \subset \partial M.$ However, this implies that $f^{-1}(t)\cap N=\emptyset,$ or equivalently, $N\subset f^{-1}((-\infty,t)).$ It is now easy to check that $N\subset M$ is both an  open and a closed subset, which contradicts the fact that $M$ is connected. Therefore,  $H_n(M^t,\partial M)=0$, and hence  $H_{n-1}(\partial M) \rightarrow H_{n-1}(M^t)$ is an injection, i.e. $[\Sigma_1],\ldots , [\Sigma_l]$ are linearly independent. This further implies that $d_1,\ldots , d_k$ are linearly independent.
\\
\\
Classes $i_* d_1, \ldots , i_* d_k\in H_{n-1}(M)$ satisfy
$$i_* d_1 + \ldots + i_* d_k=i_*[\Sigma_1]+\ldots + i_*[\Sigma_l]=[\partial M]=0.$$
By using the exactness of the following part of the long exact sequence of the pair $(M,\partial M)$
$$0=H_n(M)\rightarrow  H_n(M,\partial M) \rightarrow H_{n-1}(\partial M) \rightarrow H_{n-1}(M),$$
we conclude that
$$i_* [\Sigma_1]+ \ldots + i_* [\Sigma_l]=0,$$
is the only relation which $i_*[\Sigma_1], \ldots , i_*[\Sigma_l]$ satisfy. More formally, restriction of $i_*$ to $\langle [\Sigma_1], \ldots , [\Sigma_l] \rangle_{\mathbb{R}}\subset H_{n-1}(M^t)$ has the one-dimensional kernel given by
$$\ker (i_* \vert_{\langle [\Sigma_1], \ldots , [\Sigma_l] \rangle_{\mathbb{R}}})= \langle [\Sigma_1] + \ldots + [\Sigma_l] \rangle_{\mathbb{R}}.$$
This readily implies the only relation which $i_* d_1, \ldots , i_* d_k$ satisfy is that their sum is zero, that is
$$\ker ( i_* \vert_{\langle d_1, \ldots , d_k \rangle_{\mathbb{R}}})= \langle d_1 + \ldots + d_k \rangle_{\mathbb{R}}.$$
Finally, combining the last equality with the fact that $d_1,\ldots , d_k$ are linearly independent immediately gives that any $k-1$ of $[d_1],\ldots , [d_k]$ are linearly independent. \qed
\section{Miscellaneous proofs}
\label{misc}
\subsection{Proof of Lemma \ref{lemma:semicon}}
\label{subs:semicon}
It follows directly from the definitions and non-negativity of $u$ that
$$\sum\limits_{I\in \mathcal{B'}(h)} \int\limits_I u(t)~dt\geq \sum\limits_{\tilde{I}\in \mathcal{B'}(f)} \int\limits_{\tilde{I}} u(t)~dt - 2 |\mathcal{B}'(f)|\cdot  \max_{[\min f, \max f]} u  \cdot d_{bottle}(\mathcal{B}(f),\mathcal{B}(h)),$$
which means that we are left to prove that
\begin{equation}\label{CASE1}
\int\limits_{\min f}^{\max f} u(t)~dt - \int\limits_{\min h}^{\max h} u(t)~dt \leq 2\cdot  \max_{[\min f, \max f]} u  \cdot d_{bottle}(\mathcal{B}(f),\mathcal{B}(h)),
\end{equation}
if $\partial M= \emptyset$ and
\begin{equation}\label{CASE2}
\int\limits_{\min f}^0 u(t)~dt - \int\limits_{\min h}^0 u(t)~dt \leq \max_{[\min f, \max f]} u  \cdot d_{bottle}(\mathcal{B}(f),\mathcal{B}(h)),
\end{equation}
if $\partial M \neq \emptyset.$
Let us prove (\ref{CASE2}). If $\min f \geq \min h$ the left hand side of (\ref{CASE2}) is non-positive and hence the inequality trivially holds. If $\min f < \min h$ we need to prove that
$$\int\limits_{\min f}^{\min h} u(t)~dt \leq \max_{[\min f, \max f]} u  \cdot d_{bottle}(\mathcal{B}(f),\mathcal{B}(h)).$$
However, in every $(d_{bottle}(\mathcal{B}(f),\mathcal{B}(h))+\varepsilon)$-matching the infinite bar $(\min f, +\infty)\in \mathcal{B}(f)$ has to be matched with some infinite bar $(a,+\infty) \in \mathcal{B}(h)$ and since $\min h$ is the smallest of all endpoints of all infinite bars in $\mathcal{B}(h),$ we have that
$$\min h - \min f \leq a - \min f \leq d_{bottle}(\mathcal{B}(f),\mathcal{B}(h))+\varepsilon.$$
Since the above holds for all $\varepsilon > 0$ the inequality is proven. To prove (\ref{CASE1}) one proceeds in the similar fashion, by analysing cases depending on the relative position of $\min f, \min h$ and $\max f, \max h.$ This completes the proof of the lemma. \qed
\subsection{Proof of Lemma \ref{Lipsch}}
\label{lemmaproof}
We prove the statement in the case of $M$ without boundary, the other case is treated the same way. Let $\mathcal{B}(f)$ and $\mathcal{B}(\tilde{f})$ be two barcodes associated to two Morse functions and denote finite bars by $I_i\in \mathcal{B}(f)$, $\tilde{I}_j\in \mathcal{B}(\tilde{f})$ where intervals are sorted by integral of $u$ as before. Assume that $\Phi_{u,k}(\mathcal{B}(f))\geq \Phi_{u,k}(\mathcal{B}(\tilde{f}))$ and $\mu:\mathcal{B}(f)\rightarrow \mathcal{B}(\tilde{f})$ is an $\varepsilon$-matching between these barcodes (we add bars of length 0 if needed and assume that $\mu$ is a genuine bijection). For every finite bar $I\in \mathcal{B}(f)$ we have that the distance between endpoints of $I$ and $\mu(I)\in \mathcal{B}(\tilde{f})$ is less or equal than $\varepsilon$ and hence
$$\bigg| \int_I u(t)~dt-\int_{\mu(I)}u(t)~dt \bigg| \leq 2 \varepsilon \max u.$$
Also $|\min f-\min \tilde{f}|\leq \varepsilon$ and $|\max f-\max \tilde{f}|\leq \varepsilon$ and hence
$$\bigg| \int\limits_{\min f}^{\max f} u(t)~dt - \int\limits_{\min \tilde{f}}^{\max \tilde{f}} u(t)~dt \bigg| \leq 2\varepsilon \max u. $$
Using these estimates and the fact that the integrals of $u$ over $\tilde{I}_j$ decrease with $j$ we get
$$0 \leq \Phi_{u,k}(\mathcal{B}(f)) - \Phi_{u,k}(\mathcal{B}(\tilde{f}))=\int\limits_{\min f}^{\max f} u(t)~dt - \int\limits_{\min \tilde{f}}^{\max \tilde{f}} u(t)~dt + \sum\limits_{i=1}^k \int\limits_{I_i} u(t)~dt - $$
$$- \sum\limits_{j=1}^k \int\limits_{\tilde{I}_j} u(t)~dt \leq2\varepsilon \max u + \sum\limits_{i=1}^k \int\limits_{I_i} u(t)~dt - \sum\limits_{i=1}^k \int\limits_{\mu(I_i)} u(t)~dt  \leq \varepsilon(2k+2)\max u.$$  Taking infimum over all $\varepsilon$-matchings finishes the proof. \qed
\subsection{Proof of Proposition \ref{prop:torus}}
\label{subs:torus}
The barcode $\mathcal{B}(f_l)$ of the function
$$f_l(x_1,\ldots , x_n)=\frac{\sqrt{2}}{n(2\pi)^{\frac{n}{2}}}(\sin lx_1 + \ldots + \sin l x_n),~l\in \mathbb{N}.$$
can be computed using  the K\"unneth formula for persistence modules. Below we briefly explain how to apply  this formula and refer the reader to \cite{PSS} for a more detailed treatment.

Given two Morse functions $f:M_1\rightarrow \mathbb{R}$ and $h:M_2\rightarrow \mathbb{R}$ we define another Morse function $f\oplus h:M_1 \times M_2 \rightarrow \mathbb{R}$ by setting $f\oplus h(x_1,x_2)=f(x_1)+h(x_2).$ Barcode $\mathcal{B}(f\oplus h)$ may be computed from $\mathcal{B}(f)$ and $\mathcal{B}(h)$ via the following procedure:
\begin{itemize}
\item An infinite bar $(a,+\infty)\in \mathcal{B}_i(f)$ and an infinite bar $(c,+\infty)\in \mathcal{B}_j(h)$ produce an infinite bar $(a+c,+\infty)\in \mathcal{B}_{i+j}(f\oplus h).$
\item An infinite bar $(a,+\infty)\in \mathcal{B}_i(f)$ and a finite bar $(c,d] \in \mathcal{B}_j(h)$ produce a finite bar $(a+c,a+d]\in \mathcal{B}_{i+j}(f\oplus h).$ The same bar is produced if $(c,d] \in \mathcal{B}_i(f)$ and $(a,+\infty)\in \mathcal{B}_j(h).$
\item A finite bar $(a,b]\in \mathcal{B}_i(f)$ and a finite bar $(c,d] \in \mathcal{B}_j(h)$ produce two finite bars $(a+c, \min \{ a+d,b+c \}]\in \mathcal{B}_{i+j}(f\oplus h)$ and $(\max \{ a+d,b+c \},b+d] \in \mathcal{B}_{i+j+1}(f\oplus h).$
\end{itemize}
In order to compute $\mathcal{B}(f_l)$ it is enough to compute the barcode of $\sin lx_1 + \ldots + \sin l x_n$ and rescale. In the light of the computational procedure described above we wish to look at $\sin l x:\mathbb{S}^1 \rightarrow \mathbb{R}$ and use $\mathbb{T}^n=(\mathbb{S}^1)^n.$ One readily checks that
$$\mathcal{B}_0(\sin l x)=\{ (-1,+\infty),(-1,1] \times (l-1) \},~ \mathcal{B}_1(\sin l x)=\{ (1,+\infty) \}$$
and hence $$\mathcal{B}_0(\sin l x_1 + \sin l x_2)=\{ (-2,+\infty),(-2,0] \times (l^2-1) \},$$ $$\mathcal{B}_1(\sin l x_1 + \sin l x_2)=\{ (0,+\infty) \times 2,(0,2] \times (l^2-1) \},$$
$$\mathcal{B}_2(\sin l x_1 + \sin l x_2)=\{ (2,+\infty) \}.$$
We claim that $\mathcal{B}(\sin lx_1 + \ldots + \sin l x_n)$ contains $2^n$ infinite bars and $\frac{1}{2}((2l)^n-2^n)$ finite bars. To prove the claim we use induction in $n.$ We have already checked that the statement holds for $n=1,2.$ To complete the induction step note that, in general, if $\mathcal{B}(f)$ contains $k_1$ infinite and $m_1$ finite bars, and
$\mathcal{B}(h)$ contains $k_2$ infinite and $m_2$ finite bars,  then $\mathcal{B}(f\oplus h)$ contains $k_1k_2$ infinite and $k_1m_2+m_1k_2+2m_1m_2$ finite bars. Taking $k_1=2^n$, $m_1=\frac{1}{2}((2l)^n-2^n)$ and $k_2=2$, $m_2=l-1$ yields the proof.

Finally, notice that via the described procedure an infinite bar and a bar of length $2$  produce a bar of length $2$, as well as that two  bars of length $2$ produce two  new bars of length $2$. Since we start with  $\mathcal{B}(\sin l x)$ for which all finite bars have length $2$, we conclude that all finite bars in $\mathcal{B}(\sin lx_1 + \ldots + \sin l x_n)$  have length $2$,
 and thus
\begin{equation}\label{Example_Multidimensional}
  \Phi_1(\sin lx_1 + \ldots + \sin l x_n)=2n+(2l)^n-2^n.
\end{equation}
Rescaling (\ref{Example_Multidimensional}) gives us
\begin{equation}
\label{last}
\Phi_1(f_l)=\frac{\sqrt{2}}{n(2\pi)^{\frac{n}{2}}}2^nl^n+\frac{\sqrt{2}}{n(2\pi)^{\frac{n}{2}}}(2n-2^n).
\end{equation}
Since $l=\sqrt{\lambda}$ we have that
$$\Phi_1(f_l)=A_n\lambda^\frac{n}{2}+B_n,$$
with the constants $A_n,B_n$ given explicitly by \eqref{last}. This completes the proof of Proposition \ref{prop:torus}. \qed

\subsection*{Acknowledgments} The authors are grateful to Lev Buhovsky for providing Example \ref{example:buh} that has lead to a reformulation of Conjectures \ref{higherdimconj_Weak} and \ref{higherdimconj}. The authors would like to thank Yossi Azar, Allan Pinkus and Justin Solomon for useful discussions, as well as Lev Buhovsky and Mikhail Sodin for helpful remarks on the early version of the paper. We also thank Yuliy Baryshnikov for bringing the reference \cite{CSEHM} to our attention. Part of this research was accomplished while I.P.
was supported by the Weston Visiting Professorship program at the Weizmann Institute of Science.


\begin{thebibliography}{ABCDE}

\bibitem[Ba]{Ba} S.A.~Barannikov, {\it The framed Morse complex and its invariants,} in ``Singularities and bifurcations", 93–-115, Adv. Soviet Math., 21, Amer. Math. Soc., Providence, RI, 1994.

\bibitem[BGW]{BGW}  U. Bauer, X. Ge and Y. Wang, {\it Measuring distance between Reeb graphs,} in Proceedings of the Thirtieth Annual Symposium on Computational geometry, 2014, ACM, 464-473.
\bibitem[BL]{BauerLesnick} U. Bauer and M. Lesnick, {\em  Induced matchings and the algebraic stability of persistence barcodes}, J. Comput. Geom., 6 (2015), no. 2, 162-191.
\bibitem[BMW]{BMW} U. Bauer, E. Munch and  Y. Wang, {\it Strong equivalence of the interleaving and functional distortion metrics for Reeb graphs,} 31st International Symposium on Computational Geometry, 461-475,
LIPIcs. Leibniz Int. Proc. Inform., 34, Schloss Dagstuhl. Leibniz-Zent. Inform., Wadern, 2015.
    \bibitem[BN]{BN} M. Belkin and P. Niyogi, {\em Laplacian eigenmaps for dimensionality reduction and data representation,} Neural Comput. 15  (2003), 1373-1396.
\bibitem[BMMPS]{BMMPS} P. Bendich, J.S. Marron, E. Miller, A. Pieloch and S. Skwerer, {\em Persistent homology analysis of brain artery trees,} Ann. appl. statistics 10,  (2016), 198.
\bibitem[Buh]{Buh} L. Buhovsky, {\em Private communication,} 2018.
\bibitem[CMW]{CMW} V. Cammarota, D. Marinucci and I. Wigman, {\em On the distribution of the critical values of random spherical harmonics},
J. Geom. Anal. 26 (2016), no. 4, 3252-3324.
\bibitem[CS]{CS} Y. Canzani and P. Sarnak, {\em Topology and nesting of the zero set components of monochromatic random waves}, to appear in Communications of Pure and Applied Mathematics,  arXiv:1701.00034.
    \bibitem[C]{Ca} G. Carlsson, {\em Topology and data,}  Bulletin Amer. Math. Soc. 46 (2009), 255--308.
\bibitem[CSGO]{pers-book} F. Chazal, V. de Silva, M. Glisse and S. Oudot,  {\em The structure and stability of persistence modules},  Springer, 2016.
\bibitem[CSEHM]{CSEHM} D. Cohen-Steiner, H. Edelsbrunner, J. Harer and Y. Mileyko, {\em Lipschitz functions have {$L_p$}-stable persistence}, Found. Comput. Math. 10 (2010), no. 2, 127-139.
    \bibitem[CL]{CL} R.R. Coifman and S. Lafon, {\em Diffusion maps,}  Appl. Comput. Harmon. Anal. 21, (2006), 5-30.
    \bibitem [CLRS]{CLRS} T.H. Cormen, C.E. Leiserson, R.L. Rivest, C. Stein, {\em Introduction to algorithms,} Third edition. MIT Press, Cambridge, MA, 2009
\bibitem[CH]{CH} R. Courant and D. Hilbert, {\em Methods of mathematical physics}, Vol. 1, Wiley, 1989.
\bibitem[E]{E} H. Edelsbrunner,  {\em A short course in computational geometry and topology,} Springer, 2014.
\bibitem[EH]{EH} H. Edelsbrunner and J. Harer, {\em Persistent homology-a survey,} Contemp.  Math. 453 (2008), 257-282.
\bibitem[FFP]{FFP} H. Feichtinger, H. F\"uhr and I. Pesenson, {\em Geometric space-frequency analysis on manifolds},  J. Fourier Anal. Appl. 22 (2016), no. 6, 1294-1355.
\bibitem[Ga]{Ganzburg} M. Ganzburg, {\em Multidimensional Jackson theorems}, Siber. Math. J. 22 (1981), no. 2, 223-231.
\bibitem[GW1]{GW} D. Gayet and J.-Y. Welschinger, {\em Universal components of random nodal sets},  Comm. Math. Phys. 347 (2016), no. 3, 777-797.
\bibitem[GW2]{GW2} D. Gayet and J.-Y. Welschinger, {\em Betti numbers of random nodal sets of elliptic pseudo-differential operators},  Asian J. Math. 21 (2017), no. 5, 811-839.
    \bibitem[G]{G} R. Ghrist, {\em Barcodes: the persistent topology of data,} Bulletin  Amer. Math. Soc. 45 (2008), 61-75.
\bibitem[JNT]{JNT} D. Jakobson, N. Nadirashvili and J. Toth, {\em Geometric properties of eigenfunctions}, Russ.  Math. Surv. 56 (2001), no. 6, 1085-1105.
\bibitem[Kr]{Kron} A. S. Kronrod, {\em On functions of two variables}, (in Russian), Uspekhi
Matem. Nauk (N.S.) 35 (1950), 24-134.
\bibitem[LL]{LL} F. Lin and D. Liu, {\em On the Betti numbers of level sets of solutions to elliptic equations}, Dis. Cont. Dyn. Sys. 36 (2016), no. 8, 4517-4529.
\bibitem[L1]{L1} A. Logunov, {\em Nodal sets of Laplace eigenfunctions: polynomial upper estimates of the Hausdorff measure},  Ann. of Math. (2) 187 (2018), no. 1, 221-239.
\bibitem[L2]{L2} A. Logunov, {\em Nodal sets of Laplace eigenfunctions: proof of Nadirashvili's conjecture and of the lower bound in Yau's conjecture},  Ann. of Math. (2) 187 (2018), no. 1, 241-262.
\bibitem[LM]{ML} A. Logunov and E. Malinnikova, {\em Nodal sets of Laplace eigenfunctions: estimates of the Hausdorff measure in dimension two and three},  50 years with Hardy spaces, 333-344,
Oper. Theory Adv. Appl., 261, Birkhäuser/Springer, Cham, 2018.
\bibitem[NS]{NS} F. Nazarov and M. Sodin, {\em On the number of nodal domains of random spherical harmonics}, Amer. J. Math. 131 (2009), no. 5, 1337-1357.
\bibitem[Ni]{Nic} L. Nicolaescu, {\em Critical sets of random smooth functions on compact manifolds}, Asian J. Math. 19 (2015), no. 3, 391-432.
    \bibitem[O]{Ou} S.Y. Oudot,  {\em Persistence theory: from quiver representations to data analysis,} American Mathematical Society, 2015.
\bibitem[PauSt]{PauSt} F. Pausinger and S. Steinerberger, {\em On the distribution of local extrema in Quantum Chaos,} Physics Letters A 379 (2015), 535-541.
\bibitem[Pe]{IP} I. Pesenson, {\em  Approximations in $L_p$-norms and Besov spaces on compact manifolds},  Trends in harmonic analysis and its applications, 199-209, Contemp. Math., 650, Amer. Math. Soc., Providence, RI, 2015.
\bibitem[Pi]{Pinkus} A. Pinkus, {\em  Negative theorems in approximation theory}, Amer. Math. Monthly, 110 (2003), no. 10, 900-911.
\bibitem[Po]{Po} G. Poliquin, {\em Superlevel sets and nodal extrema of Laplace-Beltrami eigenfunctions}, J. Spec. Theory, vol. 7 (2017), no. 1, 111-136.
\bibitem[PSS]{PSS} L. Polterovich, E. Shelukhin and V. Stojisavljevi\'c, {\em Persistence modules with operators in Morse and Floer theory }, Moscow Math. J, Volume 17, Issue 4, October-December 2017 pp. 757-786.
\bibitem[PS]{PolSod} L. Polterovich and M. Sodin, {\em Nodal inequalities on surfaces}, Math. Proc.
Cambridge Philos. Soc. 143 (2007), no. 2, 459-467.
\bibitem[Re]{Re} M. Reuter, {\em Hierarchical shape segmentation and registration via topological features of Laplace-Beltrami eigenfunctions}, Int. J. Comp. Vision 89 (2010), no. 2-3,
287-308.
\bibitem[Sa]{BP} R. Salem,
{\em On a theorem of Bohr and P\'al}, Bull. Amer. Math. Soc. 50, (1944). 579-580.
\bibitem[SW]{SW}  P. Sarnak and I. Wigman, {\em Topologies of nodal sets of random band limited functions},  in: Advances in the theory of automorphic forms and their L-functions, 351--365, Contemp. Math., 664, Amer. Math. Soc., Providence, RI, 2016.
    \bibitem[SOCG]{SOCG} P. Skraba, M. Ovsjanikov, F. Chazal and L. Guibas, {\em Persistence-based segmentation of deformable shapes,} in Computer Vision and Pattern Recognition Workshops (CVPRW), 45--52,  IEEE, 2010.
\bibitem[So]{S} C. Sogge, {\em Concerning the $ L^p$ norm of spectral clusters for second-order elliptic operators on compact manifolds}, J. Func. Anal. 77 (1988), no. 1, 123-138.
    \bibitem[U]{U} M. Usher, {\em Boundary depth in Floer theory and its applications to Hamiltonian dynamics and coisotropic submanifolds,} Israel J. Math. 184 (2011), 1--57.
\bibitem[UZ]{UZ} M. Usher, J. Zhang {\em Persistent homology and Floer-Novikov theory,} Geom. Topol. 20 (2016), no. 6, 3333-3430.
\bibitem[W1]{W} S. Weinberger, {\em What is... persistent homology?,} Notices Amer. Math. Soc. 58 (2011), 36-39.
\bibitem[W2]{W2} S. Weinberger, {\em Interpolation, the rudimentary geometry of Lipschitz function spaces, and geometric complexity,} Preprint, 2017.
\bibitem[Yo]{Yom} Y. Yomdin, {\em Global bounds for the Betti numbers of regular fibers of
differentiable mappings},  Topology 24 (1985), 145-152.
\bibitem[Yu]{Judin} V. A. Yudin, {\em A multidimensional Jackson theorem},  Mat. Zametki, 20 (1976), no. 3, 439-444.
\bibitem[Z1]{Z1} S. Zelditch, {\em Local and global properties of eigenfunctions}, Adv. Lect. Math. 7 (2008), 545-658.
\bibitem[Z2]{Z2} S. Zelditch, {\em Eigenfunctions and nodal sets}, Surv.   Differ. Geom. 18 (2013),  237-308.

\end{thebibliography}
\end{document}